\documentclass[reqno,twoside,11pt,english]{amsart}
\usepackage{amsmath,amsfonts,amssymb,amsthm,epsfig}
\newcommand{\RNum}[1]{\uppercase\expandafter{\romannumeral #1\relax}}
\usepackage{comment}
\usepackage{needspace}
\usepackage{color}
\usepackage[final,colorlinks,linkcolor=blue,anchorcolor=red,citecolor=blue]{hyperref}
\usepackage{graphicx}
\usepackage{titletoc}
\usepackage{amsmath,amsfonts,latexsym,amsthm,amsxtra,amssymb}
\allowdisplaybreaks
\usepackage{graphicx,float}

\newcommand{\lam}{\lambda}
\newcommand{\La}{\Lambda}

\newcommand{\ba}{\backslash}

\newcommand{\la}{\lambda}
\newcommand{\de}{\delta}

\newcommand{\oo}{\infty}

\newcommand{\non}{\nonumber}

\newcommand{\al}{\alpha}
\newcommand{\dist}{\text{\rm dist}}
\newcommand{\medint}{-\kern -,375cm\int}         
\newcommand{\medintinrigo}{-\kern -,315cm\int}
\newcommand{\wto}{\rightharpoonup}                

\newcommand{\diam}{\operatorname{diam}}

\newcommand{\cal}{\mathcal}

\newcommand{\HH}{{\mathcal H}}

\newcommand{\R}{{\mathbb R}}

\newcommand{\CP}{\mathbb{CP}}

\newcommand{\N}{{\mathbb N}}

\newcommand{\ep}{\varepsilon}
\newcommand{\su}{\subset}

\newcommand{\ri}{\rightarrow}

\def\supp{\textup{supp}}

\def\loc{\text{\rm loc}}

\newcommand{\sing}{\text{\rm sing}}

\newcommand{\grad}{\text{\rm grad}}

\newcommand{\Span}{\text{\rm span}}

\newcommand{\id}{\mathrm{id}}

\newcommand{\e}{\varepsilon}
\newcommand{\Div}{\mathrm{Div}}

\newcommand{\vol}{\mathrm{vol}}

\numberwithin{equation}{section}
\newtheorem{theorem}{Theorem}[section]

\newtheorem{corollary}[theorem]{Corollary}

\newtheorem{definition}[theorem]{Definition}

\newtheorem{lemma}[theorem]{Lemma}

\newtheorem{proposition}[theorem]{Proposition}
\newtheorem{remark}[theorem]{Remark}
\newtheorem*{remark*}{Remark}

\theoremstyle{definition}
\newtheorem{Thm}{Theorem}

\makeatother

\begin{document}
	
	\title[Quantitative stratification for stationary $J$-holomorphic maps]{Quantitative stratification and optimal regularity for stationary $J$-holomorphic maps} 
	
	\author[C.-Y. Guo, M.-L. Liu and C.-Y. Wang]{Chang-Yu Guo, Ming-Lun Liu and Chang-You Wang}

	\address[Chang-Yu Guo]{Research Center for Mathematics and Interdisciplinary Sciences, Shandong University 266237,  Qingdao, P. R. China and Department of Physics and Mathematics, University of Eastern Finland, 80101, Joensuu, Finland} \email{changyu.guo@sdu.edu.cn}
	
	\address[Ming-Lun Liu]{Research Center for Mathematics and Interdisciplinary Sciences, Shandong University 266237,  Qingdao, P. R. China and and  Frontiers Science Center for Nonlinear Expectations, Ministry of Education, P. R. China}
	\email{minglunliu2021@163.com}
	
	\address[Changyou Wang]{Department of Mathematics, Purdue University, West Lafayette, IN 47907, USA} \email{wang2482@purdue.edu}

	\begin{abstract}
		In this paper, we give {an alternate} proof of the $\e$-regularity theorem of Wang [Comm. Cont. Math. 2003] for weakly $J$-holomorphic maps by adapting the conservation law approach of Rivi\`ere [Invent. Math. 2007]. Then we adapt the powerful quantitative stratification method of Naber--Valtorta [Ann. Math. 2017] to obtain optimal volume estimates and rectifiability for the quantitative strata of stationary $J$-holomorphic maps.  {In contrast with general} harmonic maps, the almost complex equation  { gives rise to} an additional symmetry improvement: Every odd-dimensional translation symmetry extends to an even-dimensional one. Consequently, consecutive even and odd strata
		coincide, and the $k$-th quantitative stratum satisfies a tubular-neighborhood
		estimate with exponent $2m-2\lfloor k/2\rfloor$, where $2m$ is the real
		dimension of the domain. The corresponding strata are
		$2\lfloor k/2\rfloor$-rectifiable, with almost-everywhere unique  { complex planes of symmetry}. Finally, we prove that every stationary $J$-holomorphic
		map is smooth if the target { manifold} admits no non-constant $J$-holomorphic sphere,
		{ improving} the corresponding  { theorem} of Wang.
	\end{abstract}
	
	\date{}
	\subjclass[2020]{53C43, 58C10, 58E20}
	\keywords{$J$ holomorphic map, Regularity theory, Quantitative stratification, Singular set, Rectifiable set}
	\thanks{$^*$Corresponding author: Ming-Lun Liu}

	\maketitle
	\begingroup\small\setlength{\baselineskip}{11pt}
	\tableofcontents
	\endgroup
	\clearpage
	
	\section{Introduction}
	
	\subsection{Background and motivation}
	Pseudoholomorphic curves, introduced by M. Gromov in his celebrated work \cite{Gromov-85}, play  fundamental {roles} in symplectic geometry \cite{Wolfson-1988,Ye-1994}. Their  higher dimensional extension, pseudoholomorphic maps, are central objects in  geometry. They also form a natural extension of holomorphic maps to { almost complex} manifolds
; see \cite{Li-Sheng-2019,Miller-2015,Riviere-Tian-2004} for more about pseudoholomorphic maps. 
	
	From the analytic point of view, { it is} an important question  to understand the regularity of weakly pseudoholomorphic maps and the structure of { their} singularities. Somewhat surprisingly, although their defining equation is of first order, weak solutions with
	finite Dirichlet  {energies} exhibit  features similar to that of harmonic maps  {between} manifolds, as was shown in  { the previous} work 
    {by} Wang \cite{Wang-03-CCM}.
	
	Let $(M^{2m},g,j)$ and $(N^{2n},h,J)$ be closed smooth almost Hermitian manifolds 
    { without boundaries}, i.e. $j^2=-\id$, $J^2=-\id$, and the Riemannian metrics are compatible with their almost complex structures. Fix an isometric embedding $N\hookrightarrow\R^L$ { for some large $L\ge 1$} and set
	\[
	H^1(M,N)=\Big\{u\in H^1(M,\R^L):u(x)\in N\text{ for a.e. }x\in M\Big\}.
	\]
	The Dirichlet energy of  map $u\in H^{1}(M,N)$ is defined by 
	\[
	E(u)=\int_M|du|^2\,d\vol_g.
	\]
	
	\begin{definition}\label{def: J holo}
		A map $u\in H^1(M,N)$ is called weakly $(j,J)$-holomorphic, or simply
		weakly $J$-holomorphic, if
		\begin{equation}\label{eq: def of J holo}
			du\bigl(j(x)X\bigr)=J(u(x))\,du(X)
			\qquad\text{for a.e. }x\in M\text{ and every }X\in T_xM.
		\end{equation}
	\end{definition}
	
	When $m=1$, every weakly $J$-holomorphic map is smooth; see Ye~\cite{Ye-1994}. In higher dimensions, however, singularity may occur. For instance, the projection map $\pi:B_1\subset \mathbb{C}^m\to \mathbb{CP}^{m-1}$ $(m\geq 2)$, $\pi(z):=[z_1;\cdots; z_m]$ is a weakly $J$-holomorphic map with an isolated singularity at the origin. Moreover, a $J$-holomorphic map between  almost
	Hermitian manifolds need not be harmonic. Thus the regularity theory of harmonic maps cannot be applied directly without exploring further the analytic structure supplied by~\eqref{eq: def of J holo}.
	
	The  {partial} regularity (and blow-up) theory for weakly $J$-holomorphic maps was 
    { previously studied} by
	Wang~\cite{Wang-03-CCM}. A starting point is the following $\e$-regularity theorem, which is the analogue of  {that} for harmonic maps. Recall { that for $U\subset M$, the gradient $du$, for $u\in H^1_{\rm{loc}}(U, N)$, belongs to the Morrey { space $M^{2, 2m-2}(U)$, 
    if
	\[
	\big\|du\big\|_{M^{2,2m-2}(U)}^2
	=\sup_{x\in U,\,0<\rho<\diam(U)}\Big\{
	\rho^{2-2m}\int_{B_\rho(x)\cap U}|du|^2\,d\vol_g\Big\}<+\infty.
	\]}}
	The following  partial regularity { theorem} was obtained by Wang \cite{Wang-03-CCM}.  
	\begin{Thm}[{\cite[Theorem 1.1]{Wang-03-CCM}}]
		\label{thm: ep-regularity}
		There exists $\varepsilon_0>0$, depending only on the almost Hermitian
		geometry of $M$ and $N$, such that the following holds: If
		$u\in H^1(M,N)$ is weakly $J$-holomorphic and { for $B_r(p)\subset M$},
		\[
		\|du\|_{M^{2,2m-2}(B_r(p))}<\varepsilon_0,
		\]
		then $u$ is smooth in $B_{r/2}(p)$.
	\end{Thm}
	
	In this paper, we revisit this result through the antisymmetric structure and conservation law of the equation,  { developed} by Rivi\`ere \cite{Riviere-2007-Invent} in his celebrated work on { the} conservation law of harmonic maps into general target manifolds. Extend $J$ to an endomorphism $\widetilde J$ of the ambient Euclidean bundle by setting it equal to zero on the normal
	bundle of $N$. The compatibility of $J$ with $h$ implies that
	$\widetilde J$ is skew-symmetric. As we shall prove in
	Section~\ref{subsec:mono for and partial regularity}, a weakly
	$J$-holomorphic map $u$ satisfies the conservation law:
	$$
	\Div_g\bigl(\widetilde{J}(u)\,\grad_g u-u\otimes\Div_g(g^{-1}j)\bigr)=0.
	$$
	Based on this conservation law, we infer that $u$ satisfies an elliptic system of Rivi\`ere type:
	\begin{equation}\label{eq: intro antisymmetric system}
		-\Delta_g u=\langle\Omega,\grad_g u\rangle_g,
		\qquad \Omega\in L^2(TM\otimes\mathfrak{so}(L)),
	\end{equation}
	with growth condition:
    $$|\Omega|\leq C(|du|+1).$$ 
    Here vectors and co-vectors are identified using $g$. This allows us to
	apply the regularity theorem of Rivi\`ere--Struwe~\cite{RS-2008} to give  { an alternate} proof of  Theorem~\ref{thm: ep-regularity}; see Theorem \ref{thm:new ep-regularity} below. 
	
	To obtain quantitative  { properties on the} singularities, we need to consider the class of { stationary solutions}.
	\begin{definition}\label{def: stationary J holo}
		A weakly $J$-holomorphic map $u\in H^1(M,N)$ is called stationary if
		\begin{equation}\label{eq: stationary J holo}
			\left.\frac{d}{dt}\right|_{t=0} E(u\circ F_t)=0
		\end{equation}
		{ holds} for every smooth compactly supported vector field $X$ on $M$, where
		{ $F_t: M\to M$ is the local smooth flow map generated by $X$.}
	\end{definition}
	
	As in the case of harmonic maps, stationarity yields a monotonicity inequality for the normalized
	Dirichlet energy; see \cite[Proposition 3.1]{Wang-03-CCM}. Together with Theorem~\ref{thm: ep-regularity}, this gives the partial regularity theory for stationary {$J$-}holomorphic maps: The map $u$ is smooth outside the singular set 
	\[
	\sing(u)=\Big\{x\in M:u\text{ is not smooth in any neighborhood of }x\Big\}
	\]
	with $\mathcal H^{2m-2}(\sing(u))=0$; see~\cite[Theorem 1.3]{Wang-03-CCM}.
	Adapting the seminal work of Lin \cite{Lin-1999-Annals} on stationary harmonic maps, Wang \cite[Theorem 1.4]{Wang-03-CCM} also established a blow-up theory describing the defect measure of
	a weakly convergent sequence of stationary $J$-holomorphic maps. In
	particular, the concentration set is $(2m-2)$-rectifiable with almost complex
	tangent planes and 
    { the} energy concentration produces a
	$J$-holomorphic sphere in the target {manifold}.
	
	On the other hand, it is an interesting open question whether we can  { extract further} geometric information about the {singular sets}. { According to the powerful techniques developed by Nabor-Valtorta 
    \cite{Naber-V-2017}},  a standard way to investigate these questions is to
	classify {singular} points according to the symmetries of their tangent maps. For $x\in M$ and $r>0$, consider the rescaled map
	\[
	u_{x,r}(y)=u(\exp_x(ry)).
	\]
	{ Any sub-sequential weak $H^1$ limit as $r\downarrow0$ is called a tangent
	map of $u$ at $x$. Such a map is homogeneous of degree zero and satisfies} the
	$J$-holomorphic equation on $T_xM$ with the constant domain complex
	structure $j_x$. Recall that a map on $T_xM$ is called $k$-symmetric if it is
	homogeneous of degree zero and invariant under translations along a
	$k$-dimensional real { vector} subspace. This leads to the  stratification of the singular
	stratum
	\[
	S^k(u)=\Big\{x\in M:\text{ no tangent map of }u\text{ at }x
	\text{ is }(k+1)\text{-symmetric}\Big\}.
	\]
	
	For harmonic maps, Cheeger--Naber~\cite{Cheeger-Naber-2013-CPAM}
	introduced quantitative stratification to measure the failure of symmetry
	on a prescribed range of scales and established almost optimal regularity for harmonic maps. Then in a 
    remarkable work, Naber--Valtorta~\cite{Naber-V-2017} developed a quantitative stratification method and  obtained optimal volume estimates about the Tubular neighborhood
	of each quantitative stratum and proved the corresponding rectifiability. Their
	approach combines quantitative cone splitting, an $L^2$ subspace
	approximation theorem, the quantitative Reifenberg theorems and a  delicate covering argument. Their method turns out to be very robust and has been successfully applied to many other geometric problems; see \cite{Cheeger-N-V-2015,Cheeger-J-N-2021,Fu-Wang-Zhang-2024,GJXZ-2024,Guo-Liu-Xiang-2026,Guo-Sun-2026,Naber-V-2020-JEMS,Wang-Wei-Wu-2025} for more information.
	
	For stationary $J$-holomorphic maps, there is a further geometric feature
	that should be reflected in the stratification. If a tangent map $v$ is
	translation invariant in a direction $a\in T_xM$, then
	\[
	dv(j_xa)=J(v)\,dv(a)=0.
	\]
	Consequently, its full space of translation symmetries is $j_x$-invariant
	and has even real dimension. In particular, a $(2\ell-1)$-symmetric
	tangent map is automatically $2\ell$-symmetric; see Lemma \ref{lemma: symmetry improvement} below. It is therefore natural
	to ask whether the quantitative stratification theory can capture this
	additional complex symmetry.
	
	\medskip
	\noindent\textbf{Question.}
	\emph{Do the singular strata of stationary $J$-holomorphic maps necessarily
		rectifiable? Can we prove optimal volume estimates for the singular strata?}
	\medskip
	
	A second question concerns the role of holomorphic spheres in regularity.
	Under the assumption that $(N,J)$ contains no non-constant
	$J$-holomorphic sphere, Wang~\cite[Theorem 1.6]{Wang-03-CCM} proved
	the improved estimate
	$\dim_{\mathcal H}\sing(u)\leq2m-6$, with a discrete singular set in
	real dimension six {(i.e., $m=3$)}. Since the same assumption rules out the defect
	measure in the compactness theory, it is natural to ask whether it
	actually excludes all singularities, in every dimension.
	
	\subsection{Main results}
	
	Our purpose of this paper is to provide affirmative answers to these two questions. Our first
	result extends the quantitative stratification theory to stationary
	$J$-holomorphic maps and incorporates the symmetry improvement described
	above. Our second result establishes smoothness when the target
	admits no non-constant $J$-holomorphic sphere.
	
	We briefly recall the notation for quantitative strata. The detailed
	definitions are given in Section~\ref{sec:quantitative ep regularity}.
	Roughly speaking, $u$ is $(k,\varepsilon)$-symmetric on a ball $B_s(x)$
	if $u_{x,s}$ is $\varepsilon$-close in the averaged squared $L^2$-distance (on the unit ball) to a $k$-symmetric map; see Definition \ref{def: quantitative symmetry} for the precise meaning. For $0<r<1$, set
	\[
	S^k_{\varepsilon,r}(u)=\Big\{x\in M:
	\ B_s(x)\text{ is not }(k+1,\varepsilon)\text{-symmetric}
	\ \text{for every }s\in[r,1)\Big\},
	\]
	and define
	\[
	S^k_\varepsilon(u)=\bigcap_{r>0}S^k_{\varepsilon,r}(u),
	\qquad
	S^k(u)=\bigcup_{i=1}^{\infty}S^k_{1/i}(u).
	\]
	The latter agrees with the tangent-map definition above; see
	Lemma~\ref{lemma: stratification for SJholo}. For a set $A\subset M$,
	write $T_r(A)=\big\{x\in M:\dist_g(x,A)<r\big\}$.
	
	To state our main result, we denote the class of maps with bounded energy 
	by
	$$ H^1_\Lambda(M,N)=\Big\{u\in H^1(M,N):E(u)\leq\Lambda\Big\}$$
    {for $\Lambda>0$}.
	Our main result of this paper provides an affirmative answer to the previous Question. 
	\begin{theorem}[Optimal volume estimates and rectifiability]
		\label{thm: stratification of SJholo}
		{Given any $\Lambda>0$,} let $u\in H^1_\Lambda(M^{2m},N^{2n})$ be a stationary
		$J$-holomorphic map, and let $0\leq k\leq2m-1$. Set
		$d_k=2\lfloor k/2\rfloor${, where $\lfloor x\rfloor$ denotes the integer part of $x$.}
		
		\smallskip
		\noindent{\upshape(i)}
		For every $\varepsilon>0$, there is a constant $C_\varepsilon=C(m,M,N,g,h,j,J,\Lambda,\varepsilon)$ such that, for
		every $p\in M$ and $0<r<1$, the following estimates hold: If $k$ is even,
		\begin{equation}\label{eq: Vol estomate-1}
			\vol_g\bigl(T_r(S^k_{\varepsilon,r}(u))\cap B_1(p)\bigr)
			\leq C_\varepsilon r^{2m-k},
		\end{equation}
		{ and if} $k$ is odd, there exists
		$\delta_\varepsilon=\delta(m,M,N,g,h,j,J,\Lambda,\varepsilon)>0$
		such that $S^k_{\varepsilon,r}(u)\subset
		S^{k-1}_{\delta_\varepsilon,r}(u)$ and
		\begin{equation}\label{eq: Vol estomate-1-1}
			\vol_g\bigl(T_r(S^k_{\varepsilon,r}(u))\cap B_1(p)\bigr)
			\leq C_\varepsilon r^{2m-k+1}.
		\end{equation}
		Consequently,
		\begin{align}
			\vol_g\bigl(T_r(S^k_\varepsilon(u))\cap B_1(p)\bigr)
			&\leq C_\varepsilon r^{2m-k}
			&&\text{if }k\text{ is even},\label{eq: Vol estomate-2}\\
			\vol_g\bigl(T_r(S^k_\varepsilon(u))\cap B_1(p)\bigr)
			&\leq C_\varepsilon r^{2m-k+1}
			&&\text{if }k\text{ is odd}.\label{eq: Vol estomate-2-1}
		\end{align}
		
		\smallskip
		\noindent{\upshape(ii)}
		The sets $S^k_\varepsilon(u)$ and $S^k(u)$ are countably
		$d_k$-rectifiable, and $S^k_\varepsilon(u)$ is locally upper Ahlfors
		$d_k$-regular. Moreover,
		\[
		S^{2\ell-2}(u)=S^{2\ell-1}(u),\qquad 1\leq\ell\leq m.
		\]
		For $\mathcal H^{d_k}$-almost every $x\in S^k(u)=S^{d_k}(u)$,
		there is a unique $j_x$-invariant real $d_k$-plane
		$V_x^{d_k}\subset T_xM$ such that every tangent map of $u$ at $x$
		is $d_k$-symmetric with respect to $V_x^{d_k}$.
	\end{theorem}
	
	Here { a set $E$ is} countable $d$-rectifiable, { if}, up to a
	$\mathcal H^d$-null set, the set $E$ is covered by countably many Lipschitz
	images of subsets of $\R^d$. The local upper Ahlfors  { $d_k$-regular} in the
	theorem  { refers to} 
	\[
	\mathcal H^{d_k}\bigl(S^k_\varepsilon(u)\cap B_\rho(x)\bigr)
	\leq C_\varepsilon\rho^{d_k}, \ \rho>0.
	\]
	The estimates in Theorem~\ref{thm: stratification of SJholo} have the
	dimensional order $r^{2m-d_k}$. For odd $k$, the almost complex equation improves the exponent
	by one compared with the estimate associated with a general real
	$k$-dimensional stratum. In particular, locally on $M$,
	\[
	\overline{\dim}_{\rm Min}S^k_\varepsilon(u)\leq d_k,
	\qquad
	\dim_{\mathcal H}S^k(u)\leq d_k.
	\]
	
	Our second main result 
    { improves} the regularity conclusion of
	\cite[Theorem 1.6]{Wang-03-CCM} by showing that the absence of
	holomorphic spheres directly implies smoothness, which is  different  { from } the case of stationary harmonic maps \cite{Naber-V-2017}.
	
	\begin{theorem}[Regularity without holomorphic spheres]
		\label{thm: N no holo S^2 implies u smooth}
		{For $m\ge 2$,} let $u\in H^1(M^{2m},N)$ be a stationary
		$J$-holomorphic map.  If every smooth $J$-holomorphic map
		$S^2\cong\CP^1\to N$ is constant, then $u\in C^\infty(M,N)$.
	\end{theorem}
	
	By \cite[Proposition 4.1]{Barbaro-Fagioli-RiosOrtiz-2022}, the complex tori $N=\mathbb C^n/\Gamma$, where $\Gamma$ is a lattice of rank $2n$, with its standard complex structure admits no nonconstant smooth holomorphic spheres. More generally, any compact complex submanifold $\widehat{N}$ of the complex torus $N$ (with its induced complex structure) admits no nonconstant smooth holomorphic spheres. Indeed, let $i:\widehat{N}\hookrightarrow N$ be a compact complex submanifold of $N$. For any holomorphic map $f:\CP^1\to \widehat{N}$, by \cite[Proposition 4.1]{Barbaro-Fagioli-RiosOrtiz-2022}, $i\circ f:\CP^1\to N$ is constant. Hence $f$ is constant.
	
	We next explain the main ideas of the proofs. For
	Theorem~\ref{thm: stratification of SJholo}, the analytic starting
	points are the $\e$-regularity theorem, the monotonicity
	inequality, the compactness and defect-measure theory.  { Since there may be loss of Dirichlet energy under weak convergences},  it is necessary to
	keep track of both tangent maps and  { the defect} measures. Small drops
	of the normalized energy control radial derivatives and thus lead to
	quantitative homogeneity. Cone splitting then detects additional
	translation symmetries, while the $J$-holomorphic equation supplies the
	even-dimensional symmetry improvement.
	
	The next step is an $L^2$ subspace approximation estimate, which bounds
	Jones' $\beta$-numbers by the radial energy on suitable annuli. This
	connects the analytic information from monotonicity to the geometric
	Reifenberg criteria of Naber--Valtorta~\cite{Naber-V-2017}. A  standard covering argument, due to Naber-Valtorta \cite{Naber-V-2017,Naber-V-2018}, gives the volume estimates and rectifiability. The main task is to establish these
	ingredients for stationary $J$-holomorphic maps while retaining the
	additional information  { induced} by the almost complex structures.
	
	The proof of Theorem~\ref{thm: N no holo S^2 implies u smooth} uses
	the complex symmetry in a different way. The absence of holomorphic
	spheres eliminates energy loss in blow-up sequences. A nonconstant
	homogeneous tangent map at a singular point is invariant in the radial
	direction and, by the $J$-holomorphic equation, in the corresponding
	circular direction. Then it reduces to complex projective
	space. The argument proceeds by induction on the complex dimension:
	regularity of the descended map permits restriction to a projective
	line, producing a nonconstant $J$-holomorphic sphere, which contradicts
	the assumption on the target.
	
	\medskip
	\noindent\textbf{Organization of this paper.}
	In Section~\ref{sec:Basic Properties}, we discuss $\e$-regularity, monotonicity, compactness, unique continuation and tangent
	maps. In Section~\ref{sec:quantitative ep regularity}, we develops quantitative
	symmetry, cone splitting, and the symmetry improvement for the strata.
	In Section~\ref{sec:reifenberg}, we recall the quantitative Reifenberg
	theorems and establish the $L^2$-subspace approximation estimate.
	Section~\ref{sec:covering lemma} contains the needed covering arguments.
	In the final section, Section~\ref{sec:proofs}, we prove  the main theorems.
	
	\medskip
	\noindent\textbf{Notation of this paper.}
	The real dimension of the domain is $2m$. We write $B_r(x)$ for a
	geodesic ball, $\vol_g$ for Riemannian volume, and $\mathcal H^k$ for
	$k$-dimensional Hausdorff measure. The notation $D$ or $\partial$ is used
	for coordinate derivatives, while $\nabla$ denotes covariant
	differentiation. In local arguments, the metric and almost complex
	structure are pulled back by the exponential map.  Unless otherwise specified, constants
	depend only on the indicated energy bound and the fixed almost
	Hermitian geometry of the domain and target.

	\section{Basic Properties}\label{sec:Basic Properties}

	\subsection{Monotonicity formula and partial regularity}\label{subsec:mono for and partial regularity}
	
	The following characterization of stationarity is well-known. 
	\begin{proposition}[\cite{Price-83}]
		\label{prop: stationary condition}
		A weakly $J$-holomorphic map $u\in H^1(M,N)$ is stationary if and only if
		\begin{equation}\label{eq: stationary condition}
			\int_M |Du|^2 \Div_g(X) - 2 \langle du(\nabla^M_{e_i} X), du(e_i) \rangle = 0,
		\end{equation}
		for any smooth vector field $X$ with compact support. Here $\nabla^M$ denotes the Levi–Civita connection on $M$, $\Div_g$ denotes the divergence on $M$ and $\{e_i\}_{i=1}^{2m}$ is an orthonormal { tangent} frame field { on $M$} with respect to $g$. 
	\end{proposition}
	
	Our starting point of analysis is the monotonicity inequality.
	\begin{proposition}\cite[Proposition 3.1]{Wang-03-CCM}\label{prop: monotonic}
		For $m\ge 2$. Let $u \in H^1(M, N)$ be a stationary $J$-holomorphic map. Then there exist $R_0, C_0 > 0$ depending only on $(M, g)$ such that for any $a \in M$, $0 < r \leq R \leq R_0$, there holds
		\begin{equation}\label{eq: monotonic}
			e^{C_0 R} R^{2-2m} \int_{B_R(a)} |Du|^2 - e^{C_0 r} r^{2-2m} \int_{B_r(a)} |Du|^2 \geq 2 \int_{B_R(a) \setminus B_r(a)} \rho^{2-2m} \left| \partial_\rho u \right|^2,
		\end{equation}
		where $\rho=|x-a|$, $\partial_{\rho}u=Du\cdot \frac{x-a}{|x-a|}$. 
	\end{proposition}
	
	For our later application, we introduce
	\begin{equation}\label{eq: eDJ energy}
		\Phi_{u}(p, r)=e^{C_0r}r^{2-m}\int_{B_r(p)}|D{ u}|^2dx,
	\end{equation}
	\begin{equation}\label{eq: nablaJ energy}
		\Phi^g_{u}(p, r)=r^{2-m}\int_{B_r(p)}|\nabla { u}|^2d\vol_g,
	\end{equation}
	and
	\begin{equation}\label{eq: DJ energy}
		\Phi^\delta_{u}(p, r)=r^{2-m}\int_{B_r(p)}|D{ u}|^2dx.
	\end{equation}
	By the monotonicity formula \eqref{eq: monotonic}, we know that $\Phi_u(x,r)$ is increasing in $r$, and hence $\lim_{r\downarrow 0}\Phi_u(x,r)$ exists.
	\begin{definition}[Density function]\label{def:density function}
    { Let $u \in H^1(M, N)$ be a stationary $J$-holomorphic map.}
		For $x\in M$, we define the density function { of $u$ at $x$}  as
		$$
		\Phi_u(x):=\lim\limits_{r\downarrow0}\Phi_u(x,r)=\lim\limits_{r\downarrow0}\Phi^g_u(x,r)=\lim\limits_{r\downarrow0}\Phi^\delta_u(x,r).
		$$
	\end{definition}

	To state our partial regularity theorem below, we first recall the definition of Morrey space. Let $1\le p<\infty$ and $0\le s\le n$. The Morrey space $M^{p,s}(U)$ consists of functions $f\in L^{p}(U)$ such that
	$$
	\big\|f\big\|_{M^{p,s}(U)}\equiv\sup_{x\in U,0<r<\mathrm{diam}(U)}r^{-s/p}\|f\|_{L^{p}(B_{r}(x)\cap U)}<\infty.
	$$
	
	The main partial regularity theorem proved by Wang \cite{Wang-03-CCM} reads as follows. In case of harmonic maps, similar result has been obtained by Evans \cite{Evans-1991} { and Bethuel \cite{Bethuel-93}}. 
	
	Suppose that $(N,h,J)$ is isometrically embedded in $\mathbb{R}^L$. At any point $q\in N$, we extend $J$ from $T_qN$ to $T_q\R^L=T_qN\oplus(T_qN)^\perp$, denoted by $\widetilde{J}$, where for any $X\in T_q\R^L$, $$
	\widetilde{J}(X):= (J(X^{\top}),0).
	$$Here $X^\top\in T_qN$ is the tangent part of $X$. Equivalently, after an orthogonal tranformation if necessary, we have $$
	\widetilde{J}_q=\left(\begin{matrix}
		J_q&0\\
		0&0\\
	\end{matrix}\right).
	$$
	We observe that although $\widetilde{J}$ is not an almost complex structure, $\widetilde{J}$ is anti-symmetric. Indeed, for any $q\in N,X,Y\in T_q\R^L$, there holds$$
	\left\langle \widetilde{J}_q(X),Y\right\rangle=\left\langle J_q(X^\top), Y^\top\right\rangle =\left\langle X^\top, -J_q(Y^\top) \right\rangle =-\left\langle X,\widetilde{J}_q(Y) \right\rangle.  
	$$ 
	Thus, we can rewrite \eqref{eq: def of J holo}. Suppose $\{x_i\}_{i=1}^{2m}$ is a local orthonormal basis on $M$, $j=j_a^b dx^a\otimes \frac{\partial}{\partial x^b}$, $\widetilde{J}=\widetilde{J}_k^\ell dy^k\otimes \frac{\partial}{\partial y^\ell}$. Then the $J$-holomorphic equation \eqref{eq: def of J holo} becomes
	\begin{equation}\label{eq: rewrite Jholo}
		\frac{\partial u^k}{\partial x^a}=-j^i_a\widetilde{J}^k_\ell(u)\frac{\partial u^\ell}{\partial x^i}.
	\end{equation}

	\begin{theorem}\label{thm:new ep-regularity}
		Let $u\in H^1(M^{2m},N)$ be a weakly $J$-holomorphic map. Then 
		\begin{enumerate}
			\item $u$ satisfies the conservation law
			\begin{equation}\label{eq: conservation LAW in thm}
				\Div_g\bigl(\widetilde{J}(u)\,\grad_g u-u\otimes\Div_g(g^{-1}j)\bigr)=0.
			\end{equation}
			
			\item There exists $\e=\e(m,M,N,g,h,j,J)>0$ such that if for $B_r(p)\subset M$, 
            $$\|du\|_{M^{2, 2m-2}(B_r(p))}<\e, $$ then $u$ is smooth in $B_{r/2}(p)$.
		\end{enumerate}
	\end{theorem}
	
	\begin{proof}
		Throughout this proof, we use the Einstein summation convention. From now on we let $\Div_g X= \frac{1}{\sqrt{g}}\partial_i(\sqrt{g}X^i)$
         { where} $X=X^i\partial_i$, $\grad_g f=g^{ij}(\partial_jf)\partial_i$ for function $f$, and $\Delta_g=\Div_g\grad_g$. 
		Since $j$ is compatible with $g$, we have $g^{-1}j$ is anti-symmetric, i.e. $g^{ai}j_i^b=-g^{bi}j_i^a$. Then the  equation 
        { of $J$-holomorphic maps} is equivalent to 
		\begin{equation}
			\widetilde{J}(u)\,\grad_g u
            = -j\,\grad_g u.
		\end{equation}
		Thus, consider vector field $\left(u^k\partial_a(g^{ai}j_i^b)\right) \partial_b$, applying the definition of $J$-holomorphic maps we have
		$$
		\begin{aligned}
			u^k\partial_a(g^{ai}j_i^b) =& \partial_a(u^kg^{ai}j_i^b)\ 
            {-}\ \partial_a u^k g^{ai} j_i^b\\
			=& \partial_a(u^kg^{ai}j_i^b)\ {+}\ \partial_a u^k g^{bi} j_i^a\\
			=& \partial_a(u^kg^{ai}j_i^b)\ {+}\  g^{bi}\widetilde{J}^k_\ell(u)\partial_i u^\ell,
		\end{aligned}
		$$
		or equivalently,
		$$
		\Div_g\left(u\otimes g^{-1}j\right)= u \otimes\Div_g\left( g^{-1}j\right) \ {-}\ \widetilde{J}(u)\,\grad_g u.
		$$
		
		Since, for every $k=1,...,L$, $$\Div_g\Div_g(u^kg^{ai}j_i^b)=\sum_{b=1}^{2m} \frac{1}{\sqrt{g}}\partial_b\left(\sum_{a=1}^{2m}\partial_a(\sqrt{g}u^kg^{ai}j_i^b)\right)=0\quad\text{in the sense of distribution,}$$
        { thanks to $\partial^2_{ab}=\partial^2_{ba}$
        and $\sqrt{g}u^k g^{ai}j_{i}^b
        =-\sqrt{g}u^k g^{bi}j_{i}^a$,}
		we have the conservation law:
		\begin{equation}\label{eq: conservation law}
			\Div_g\bigl(\widetilde{J}(u)\,\grad_g u-u\otimes\Div_g(g^{-1}j)\bigr)=0.
		\end{equation}
		
		Based on the conservation law \eqref{eq: conservation law}, we have the following equation 
		\begin{equation}\label{eq: Delta g u equ-1}
			\begin{aligned}
				-\Delta_gu=&-\Div_g(\grad_g u)=\Div_g (\widetilde{J}(u)\,\widetilde{J}(u)\,\grad_g u)\\
				=& \left\langle \grad_g(\widetilde{J}(u)), \widetilde{J}(u)\,\grad_g u\right\rangle_g+ \widetilde{J}(u)\Div_g\left(u \otimes\Div_g(g^{-1}j)\right)\\
				=& \left\langle j\,\grad_g(\widetilde{J}(u)), \grad_g u\right\rangle_g+ \widetilde{J}(u)\Div_g\left(u \otimes\Div_g(g^{-1}j)\right).
			\end{aligned}
		\end{equation} 
		By applying $\Div_g\Div_g(g^{-1}j)=0$, we have
		\begin{equation}\label{eq: Delta g u equ-2}
			\begin{aligned}
				\widetilde{J}(u)\Div_g\left(u \otimes\Div_g(g^{-1}j)\right)= \left\langle \tilde{J}(u)\,\otimes \Div_g(g^{-1}j),\grad_gu\right\rangle_g.
			\end{aligned}
		\end{equation}
		Since $$
		\Div_g\left(\widetilde{J}(u)\otimes g^{-1}j\right)=
		j\,\grad_g\left(\widetilde{J}(u)\right) +\widetilde J(u)\otimes\Div_g(g^{-1}j),
		$$
		\eqref{eq: Delta g u equ-1} implies that
		\begin{equation}\label{eq: Delta g u equ-final}
			-\Delta_g u=\left\langle \Div_g\left(\widetilde{J}(u)\otimes g^{-1}j\right),\grad_g u\right\rangle_g.
		\end{equation}
		Define $\Omega:=\Div_g\left(\widetilde{J}(u)\otimes g^{-1}j\right)$, then \eqref{eq: Delta g u equ-final} becomes $$
		-\Delta_g u=\left\langle \Omega,\grad_g u\right\rangle_g.
		$$
		Note that $\Omega$ is an anti-symmetric  { matrix of order $L$}. Since $J,j,g$ is smooth, we have $\Omega\in L^2(M,TM\otimes\mathfrak{so}(L))$ with $$
		|\Omega|\leq C(|du|+1).
		$$
		Fix $p\in M$ and $r$ sufficiently small such that $B_{r}(p)\subset M$, we have $$\big\|\Omega\big\|_{M^{2,2m-2}(B_r(p))}\le \e.$$ The Rivi\`ere-Struwe theorem \cite[Theorem 1.1]{RS-2008}, with the smooth-coefficient extension described by Schikorra in \cite[Remark 3.3]{Schikorra-2010}, implies the local H\"older continuity of $u$.
		
		Indeed, in local coordinates, equation \eqref{eq: Delta g u equ-final} takes the form
		$$
		-\partial_a\left(\left(\sqrt{g}g^{ab}\right)\partial_bu^k\right)=\left(\sqrt{g}g^{bc}\right)\left(g_{ab}\left(\Omega_k^\ell\right)^a \right) \partial_c u^k,
		$$
		which is exactly the { same form as} in \cite[Remark 3.3]{Schikorra-2010}. 
		
		The proof { of higher order regularity} follows directly from the standard theory of elliptic equations.
	\end{proof}
	
	\begin{remark}
		The potential $\Omega$ constructed above is also divergence-free in the distributional sense. Consequently, equation \eqref{eq: Delta g u equ-final} has a div-curl structure. Thus the localized div--curl lemma by Coifman-Lions-Meyer-Semmes \cite{CLMS-1993} therefore gives $\Delta_g u\in h^1_{loc}$, where $h^1_{loc}$ denotes the local Hardy space in coordinate charts.
	\end{remark}

	In particular, similar to the theory of harmonic maps, when $m=1$, $J$-holomorphic maps are pseudo holomorphic curves, whose smoothness was established by Ye \cite{Ye-1994}.

	\begin{theorem}\cite{Ye-1994}\label{thm: 2dim smooth}
		Assume that $M$ is a compact Riemannian surface without boundary. If $u\in H^1(M,N)$ is a weakly $J$-holomorphic curve, then $u\in C^\infty(M,N)$. 
	\end{theorem}

	\subsection{Defect measure and compactness}
	
	\begin{theorem}\cite[Theorem 1.4]{Wang-03-CCM} \label{thm: defect measure}
		Let $\{u_k\}\subset H^1(M,N)$ be stationary $J$-holomorphic maps which converge weakly to $u\in H^1(M,N)$. Then $u$ is a weakly $J$-holomorphic map and there exists a closed $(2m-2)$-rectifiable set $\Sigma\subset M$, with $\mathcal H^{2m-2}(\Sigma)<\infty$, such that
		\begin{enumerate}
			\item[(i)] $u_k\to u$ in $H^1_{loc}\cap C^1_{loc}(M\setminus\Sigma,N)$ strongly. In particular, $u\in C^\infty(M\setminus\Sigma,N)$.
			\item[(ii)]
			$$
			|Du_k|^2\,dx\rightharpoonup|Du|^2\,dx+\nu
			$$
			as convergence of Radon measures, for some nonnegative Radon measure $\nu$ on $M$. Moreover, $\nu=\Theta(x)\mathcal H^{2m-2}\lfloor \Sigma$ for some nonnegative $\mathcal{H}^{2m-2}$-measurable function
			\[
			\Theta(x)=\lim_{r\to0}\frac{\nu(B_r(x))}{|B_r(x)|}\geq\e_0^2,
			\qquad\text{for $\mathcal H^{2m-2}$-a.e. }x\in\Sigma.
			\]
			\item[(iii)] For $\mathcal H^{2m-2}$-a.e. $x\in\Sigma$, $T_x\Sigma\subset\mathbb R^{2m}$ is a $j_x$-holomorphic $(2m-2)$-plane, i.e. $j_x(T_x\Sigma)=T_x\Sigma$.
			\item[(iv)] $u_k\not\to u$ in $H^1(M,N)\iff\mathcal H^{2m-2}(\Sigma)>0$. Moreover, there exists at least one { non-constant} pseudo-holomorphic $\mathbb S^2$ in $(N,J)$.
		\end{enumerate}
	\end{theorem}
	
	Based on the defect measure theorem, we have the following corollary.
	
	\begin{corollary}\label{coro: compactness of SJholo}
		Let $\{u_k\}\subset H^1(M,N)$ be stationary $J$-holomorphic maps which converge weakly to $u\in H^1(M,N)$. If $(N,J)$ does not admit any { non-constant} pseudo-holomorphic $\mathbb S^2$, then $u_k\to u$ strongly in $M$ and $u$ is a stationary $J$-holomorphic map. 
	\end{corollary}

	\subsection{Unique Continuation Property}
	
	The main result of this subsection is the following unique continuation property for stationary $J$-holomorphic maps. 
	\begin{proposition}[Unique continuation property]\label{prop: unique continuation} 
		Suppose that $u\in H^1(M^{2m},N^{2n})$ is a stationary $J$-holomorphic map on $B_1(p)\subset M$. Then $u$ enjoys the unique continuation property in the sense that if there is another weakly $J$-holomorphic map $u^\prime$ in $H^1(B_{1}(p), N)$ such that $u=u^\prime$ almost everywhere on an open set, and $u^\prime$ is smooth away from a set $\Sigma'$ of finite $\mathcal{H}^{2m-2}$-measure, then $u\equiv u^\prime$ on $B_{1}(p)$. 
	\end{proposition}
	
	For the proof of Proposition \ref{prop: unique continuation}, we need the following unique continuation theorem for smooth $J$-holomorphic maps.
	\begin{proposition}\cite{Zhang-2018}\label{prop: smooth unique continuation}
		Let $Y$ be a compact $J$-holomorphic submanifold of the almost complex manifold $(M,J)$, $u:(X,J_1)\to(M,J)$ is a $(J_1,J)$-holomorphic map with $X$ connected and compact. If $Y$ contains the image of an open subset of $X$ { under the map $u$}, then $u(X)\subset Y$.
	\end{proposition}

	\begin{proof}[Proof of Proposition \ref{prop: unique continuation}] 
		Note that if $u$ and $u'$ are both smooth, then the claim follows from proposition \ref{prop: smooth unique continuation}. Indeed, we define a map $F: (M,j)\to (N\times N, J\oplus J)$ as $F(p)=(u(p),u'(p))$. Suppose $u=u'$ on an open set $U$, then $F(U)$ is a subset of $\widetilde{N}:=\{(a,a)\in N\times N, a\in N\}$, which is a $J$-holomorphic submanifold of $(N\times N, J\oplus J)$. Thus by proposition \ref{prop: smooth unique continuation} we have $F(M)\subset \widetilde{N}$, $u\equiv u'$ on $M$.
		
		In the general case, let $\Sigma$ and $\Sigma'$ be the  { singular} set for $u$ and $u'$, respectively
. Then $u$ and $u'$ are smooth on $B_1\backslash(\Sigma\cup\Sigma')$. Since $\cal{H}^{m-2}(\Sigma\cup\Sigma')<\infty$ and thus non-disconnecting, applying proposition \ref{prop: smooth unique continuation} once again gives $u=u'$ on $B_1\backslash(\Sigma\cup\Sigma')$ and thus on the whole domain.
	\end{proof}
	
	\subsection{Tangent $J$-holomorphic maps}\label{sec: Tangent map}
	
	In this subsection, we shall explore some fundamental properties of tangent $J$-holomorphic maps. For $a\in M$ such that $B_1(a)\subset M$. Since the exponential map $\exp_a: B_1(0)\subset \R^{2m}\rightarrow B_1(a)\subset M$ is a diffeomorphism. We define the rescaled map $u_{a,r}(y):= u\circ\exp_a(ry)$ and the rescaled metric $g_{a,r}(y):= g\circ\exp_a(ry)$ for $y\in B_{1/r}(0)\subset\R^{2m}$. A basic observation is that if $u\in H^1(B_1(a),N)$ is a stationary $J$-holomorphic map, then the rescaled map $u_{a,r}(y)\in H^1(B_{1/r}(0),N)$ is also a stationary $J$-holomorphic map { with respect to the metric $g_{a,r}$}. For convenience, we still denote $u\circ\exp_a(ry)$ and $g\circ\exp_a(ry)$ by $u(a+ry)$ and $g(a+ry)$, respectively.
	\begin{definition}\label{def:tangent SJholo}
		We say that $v\in H^{1}(\mathbb{R}^{2m},N)$ is a tangent stationary $J$-holomorphic map of a stationary $J$-holomorphic map $u\in H^1(M,N)$ at the point $a\in B_1$, if there is a sequence $r_k\to 0$ such that $u_{a,r_k}:= u(a+r_k\cdot)\wto v$ in $H^1$.
	\end{definition}
	
	Based on Corollary \ref{coro: compactness of SJholo}, we are able to deduce the following  { property} for tangent stationary $J$-holomorphic maps.
	\begin{proposition} \label{prop: tangent map}
		{ For $\Lambda>0$,} let $u\in H^1_\La(M,N)$ be a stationary $J$-holomorphic map and $B_1(0)\subset M$. Suppose there is a sequence $r_k\to 0$ such that
		$$
		u_{k}\wto v\quad\text{weakly in }H^{1}
		(T_aM)$$
		$$
		|Du_k|^2\,dx\wto\mu=|Dv|^2\,dx+\nu
		$$
		 {as convergence} of Radon measures { on $T_aM=\mathbb{R}^{2m}$}, where $u_{k}=u_{a,r_k}:=u(a+r_k\cdot)$  
        { for } $a\in B_1(p)$. Then
		\begin{enumerate}
		\item[\upshape(i)] There is a closed $(2m-2)$-rectifiable set $\Sigma\subset \mathbb{R}^{2m}$ with locally finite $\mathcal{H}^{2m-2}$-measure, such that $v\in C_{\loc}^{\infty}(\mathbb{R}^{2m}\backslash \Sigma)$
		and
		$$
		u_{k}\to v\quad\text{in } C_{\loc}^{2}(\mathbb{R}^{2m}\backslash\Sigma).
		$$
		
		\item[\upshape(ii)] {\rm{(Symmetry of tangent map)}} $v$ is a weakly $J$-holomorphic map with $\sing(v)\subset\Sigma$, and moreover, it is $0$-homogeneous with respect to the origin, i.e.
		$$v(x)=v(\lambda x),\qquad\text{for all } x\in\R^{2m},\lambda>0.$$
		
		\item[\upshape(iii)] There exist a nonnegative Radon measure $\nu$ with $\supp(\nu)\cup \sing(v)=\Sigma$ and a density function $\Theta_{\nu}:\R^{2m}\to [0,\infty)$  such that
		$$
		\nu=\Theta_{\nu}{\cal H}^{2m-2}\lfloor{\Sigma}
		$$
		Moreover, there are nonnegative constants $C$ depending only on $M,N$ such that
		$$
		\e^2_{0}\le\Theta_{\nu}(x)\le C\Lambda\qquad{\cal H}^{2m-2}\text{-almost everywhere on }\Sigma.
		$$
		\item[\upshape(iv)] {\rm{(Symmetry of tangent measure)}} $\mu$ and $\nu$ are cone measures in the sense that for any $\la>0$ and any measurable set $A\subset\mathbb{R}^{2m}$, there holds
		$$\la^{2-2m}\nu(\la A)=\nu(A),\quad \la^{2-2m}\mu(\la A)=\mu(A).$$ 
        \end{enumerate}
	\end{proposition}
	
	\begin{proof}
		
		Assertions (i) and (iii) are direct consequences of Theorem \ref{thm: defect measure}. Now we prove (ii).
		Using the monotonicity formula \eqref{eq: monotonic}, we find that
		\begin{equation}\label{sinaf}
			\Phi_{u_{k}}(0,r)-\Phi_{u_{k}}(0,\rho)\geq 2\int_{B_r(0)\backslash B_\rho(0)}\frac{|x\cdot Du_{k}|^2}{|x|^{2m}}dx.
		\end{equation}
		On the other hand, for any $r>0$ we have $\Phi_{u_{k}}(0,r)=\Phi_u(a,r_kr)$ for all $k\ge 1$.
		Hence by the definition of the density function $\Theta_u$, we have
		$$\lim_{k\to\oo}\Phi_{u_{k}}(0,r)=\lim_{k\to\oo}\Phi_{u}(a,r_kr)=\lim_{r\to 0}\Phi_{u }(a,r)=\Phi_{u}(a).$$
		By letting $k\to\infty$ in \eqref{sinaf} we get
		$$
		\int_{B_r(0)\backslash B_\rho(0)} \frac{|x\cdot Dv|^2}{|x|^{2m}} dx=0.
		$$
		Since $r,\rho$ are arbitrary, this implies that $v$ is  
        { homogeneous of degree zero} with respect to the origin.	
		
		Next, we prove the symmetry of $\nu$ in assertion (iv). The method is mainly from Lin \cite{Lin-1999-Annals} and Guo-Jiang-Xiang-Zheng \cite{GJXZ-2024}. According to (ii), we have $|Dv|^2dx$ is a cone measure. Thus we only have to prove $\mu$ is a cone measure. For $\lambda>0$ and $\psi\in C_0^\infty(\mathbb R^{2m})$, set $\psi_\lambda(x)=\psi(x/\lambda)$, $\mu_{0,\lambda}(A)=\lambda^{2-2m}\mu(\lambda A)$. Since $$\begin{aligned}
			&\frac{d}{d\lambda}\left(\int_{\R^{2m}}\psi d\mu_{0,\lambda}(x)\right)
			=\frac{d}{d\lambda}\left(\lambda^{2-2m}\int_{\R^{2m}}\psi(y/\lambda) d\mu(y)\right)\\
			=&	(2-2m)\lambda^{1-2m}\int_{\R^{2m}}\psi_\lambda d\mu- \lambda^{1-2m}\int_{\R^{2m}}(y/\lambda)\cdot\nabla\psi(y/\lambda) d\mu(y).
		\end{aligned}
		$$	
		Thus, we need to prove that
		\begin{equation}\label{eq: tangent measure cone}
			\int_{\mathbb R^{2m}}
			\left((2m-2)\psi_\lambda+x\cdot D\psi_\lambda\right) d\mu=0.
		\end{equation}
		 { Choose} $X=\psi_\lambda x$ into stationary condition Proposition \ref{prop: stationary condition}, we have
		\begin{equation}\label{eq: tangent measure-01}
			\int_{\R^{2m}}\left(|Du_k|^2\Div X
			-
			2\sum_{\alpha,\beta=1}^{2m}
			\langle\partial_\alpha u_k,\partial_\beta u_k\rangle
			\partial_\alpha X^\beta\right)dx=0
		\end{equation}
		By the weak convergence of Radon measure $|Du_k|^2dx$, we have
		\begin{equation}\label{eq: tangent measure-02}	\lim\limits_{k\to\oo}\int_{\R^{2m}}\left(|Du_k|^2\Div (x\psi_\lambda)\right)dx=\int_{\R^{2m}}\left(2m\psi_\lambda+x\cdot D\psi_\lambda\right)d\mu.
		\end{equation}
		Since $X^\beta(x)=\psi_\lambda(x)x^\beta$, we have
		$$\partial_\alpha X^\beta
		=\delta_{\alpha\beta}\psi_\lambda
		+x^\beta\partial_\alpha\psi_\lambda.
		$$
		This implies
		\begin{equation}\label{eq: tangent measure-03}
			\begin{aligned}
				&\int_{\R^{2m}}\sum_{\alpha,\beta=1}^{2m}
				\langle\partial_\alpha u_k,\partial_\beta u_k\rangle
				\partial_\alpha X^\beta dx\\
				=&\int_{\R^{2m}}
				|Du_k|^2\psi_\lambda+  \big\langle\left(x\cdot Du_k\right), \left(\nabla\psi_\lambda\cdot Du_k\right)\big\rangle dx\\
				\to& \int_{\R^{2m}} \psi_\lambda d\mu
			\end{aligned}
		\end{equation}
		as $k\to\infty$. Indeed, suppose $\supp(\psi_\lambda)\subset B_R(0)$ for $R$ large enough, we have
		$$
		\begin{aligned}
			&\left|\int_{B_R(0)}\big\langle\left(x\cdot Du_k\right), \left(\nabla\psi_\lambda\cdot Du_k\right)\big\rangle dx\right|^2\\
			\le& \left(\int_{B_R(0)}|x\cdot Du_k|^2 dx\right)\|\nabla \psi_\lambda\|^2_{L^{\infty}}\|Du_k\|^2_{L^2(\supp(\psi_\lambda))}\\
			\le& C(\Lambda,\lambda)\left(\int_{B_R(0)\backslash B_r(0)}+\int_{B_r(0)}\right) |x\cdot Du_k|^2dx\\
			\le& C(\Lambda,\lambda)R^{2m}\left(\Phi_{u_k}(0,R)-\Phi_{u_k}(0,r) \right)+C(\Lambda,\lambda) r^{2m}\|Du_k\|^2_{M^{2,2m-2}(B_r)}\\
			\to& 0
		\end{aligned}
		$$as { $k\to\infty$ first and then $r\to0$}. 
		Combining \eqref{eq: tangent measure-01}, \eqref{eq: tangent measure-02} with \eqref{eq: tangent measure-03} yields \eqref{eq: tangent measure cone}. This proves (iv).	
	\end{proof}

	\section{Quantitative $\ep$-regularity theorem}\label{sec:quantitative ep regularity}

	In this section, we develop a stratification theory of the singular set for stationary $J$-holomorphic maps based on the symmetry of tangent $J$-holomorphic maps. The original idea is due to Cheeger-Naber \cite{Cheeger-Naber-2013-CPAM} and Naber-Valtorta \cite{Naber-V-2017} in their study of stationary harmonic maps. 
	
	\subsection{Quantitative symmetry and cone splitting principle}

	Consider a geodesic ball $B_{10}(p)\subset M$, we identify $B_{10}(p)$ with an Euclidean ball $B_{10}(0)$ via the exponential map, $\exp: B_{10}(0)\rightarrow M$ with $\exp(0)=p$. Again we assume that the injectivity radius of $(M, g)$ is larger than $10$, by a fixed scaling if necessary. We also assume that the metric $g$ is universally close to an Euclidean metric in $B_{5}(0)$.

	By the results in Section \ref{sec: Tangent map}, for any sequence $r_i\to0$, $u(p+r_iy)$ converges by subsequence to a tangent $J$-holomorphic map $v$ on $\R^{2m}$ which is weakly $j_0$-holomorphic map, where $j_0$ is the standard complex structure on $\R^{2m}$. Moreover, by Proposition \ref{prop: tangent map}, $v$ is homogeneous of degree zero, which gives the classical stratification of the singular set $\sing(u)$ based on homogeneity of its tangent map. For a quantitative version of such stratification, we consider the almost homogeneity property of $u$.

	\begin{definition}[Symmetry]\label{def: symmetry}
		Given a measurable map $u\colon \mathbb{R}^{\ell}\to N$, we say that
		\begin{enumerate}
		\item [\upshape(1)] $u$ is $0$-homogeneous or $0$-symmetric with respect to point $p$ if $u(p+\lambda v)=u(p+v)$ for all $\lambda>0$ and $v\in\mathbb{R}^\ell$.
		
		\item [\upshape(2)] $u$ is $k$-symmetric if it is $0$-homogeneous with respect to the origin,  and is translation invariant with respect to a $k$-dimensional subspace $V\subset\mathbb{R}^\ell$, i.e.,
		$$u(x+v)=u(x) \qquad \text{for all } x\in\R^{\ell},\, v\in V.$$
        \end{enumerate}
	\end{definition}
	
	For $u:M\to N$, we say $u$ is $0$-homogeneous with respect to $p$ if $u\circ\exp_p(\lambda v)=u\circ\exp_p(v)$. We say $u$ is $k$-symmetric with respect to $p$ if it is $0$-homogeneous with respect to $p$, and is translation invariant with respect to a $k$-dimensional subspace $V\subset\mathbb{R}^\ell$, i.e.,
	$$u\circ\exp_p(x+v)=u\circ\exp_p(x) \quad \text{for all } x\in \R^\ell,v\in V.$$
	
	If $u\in C^1(\R^\ell,N)$, then $u$ is $0$-homogeneous at $x=p$ if and only if $\partial_{r_p}u=0$, where $r_p$ is the radial direction centered at $p$; and $u$ is translation invariant with respect to a $k$-dimensional subspace $V\subset\R^\ell$ if and only if $\partial_v u=0$ for all $v\in V$.

	\begin{definition}[Quantitative  symmetry]\label{def: quantitative symmetry}
		Given a map $u\in L^2(M,N)$, $\e>0$ and nonnegative integer $k$, we say that $u$ is $(k,\ep)$-symmetric  on $B_r(x)\subset M$, or simply $B_r(x)\subset M$ is $(k,\ep)$-symmetric, if there exists some $k$-symmetric function $K\colon M\to N$ with respect to $x$ such that
		$$\medint_{B_r(x)}d_N^2(u(y),K(y))\leq\ep,$$
		where $d_N$ is the Riemannian distance on $N$. For convenience, from now on, we will use $|u(y)-K(y)|$ to represent $d_N(u(y),K(y))$.	
	\end{definition}
	
	Note that $u$ is $(k,\ep)$-symmetric on $B_r(x)$ if and only if the scaled map  $u_{x,r}(y)=u(x+ry)$ is $(k,\ep)$-symmetric on $B_1(0)$.
	
	According to the formal definition of quantitative symmetry, we may stratify the domain of a function by classifying points according to the degree of symmetry in infinitesimal neighborhoods of those points. This quantitative stratification partitions the domain into subsets where each subset corresponds to a distinct level of symmetric regularity.
	\begin{definition}[Quantitative stratification]\label{def: quantitative stratification}
		For any map $u\in L^2(M^{2m},N)$, $r,\ep>0$ and $k\in\{0,1,\cdots,2m\}$, we define the $k$-th quantitative singular stratum $S_{\ep,r}^k(u)\subset M$ as 
		$$S^k_{\ep,r}(u)\equiv\Big\{x\in M: u \text{ is not } (k+1,\ep)\text{-symmetric on }B_s(x) \text{ for any } r\le s<1\Big\}.$$
		Furthermore, we set
		$$
		S^k_\ep(u):=\bigcap_{r>0}S^k_{\ep,r}(u)
		\quad\text{ and }\quad 
		S^k(u)=\bigcup_{\ep>0}S^k_\e(u).
		$$
	\end{definition}
	
	Note that,
	$$k'\leq k\text{ or }\ep'\geq\ep\text{ or } r'\leq r\ \Longrightarrow \ S^{k'}_{\ep',r'}(u)\subseteq S^k_{\ep,r}(u).$$
	In particular, we have
	$$S^0(u)\subset S^1(u)\subset\cdots\subset S^{2m}(u)=M^{2m}.$$
	
	Now we apply this method to singular sets of stationary $J$-holomorphic maps.
	\begin{lemma}\label{lemma: stratification for SJholo} { For $\Lambda>0$,} suppose $u\in H^{1}_\La(M,N)$ is a stationary $J$-holomorphic map, then on $B_1(p)\subset M$, 
		$$S^k(u)=\Big\{x\in M: \text{no tangent $J$-holomorphic map of } u \text{ at } x \text{ is } (k+1)\mbox{-symmetric}\Big\}.$$
		Consequently, we have
		$$S^0(u)\subset S^1(u)\subset\cdots\subset S^{2m-3}(u)\subset S^{2m-2}(u)\subset S^{2m-1}(u).$$
	\end{lemma}

	\begin{proof} 
		First of all, we let
		$$\Sigma^k(u)=\Big\{x\in B_1(p): \text{no tangent $J$-holomorphic map of }u \text{ at } x \mbox{ is } (k+1)\mbox{-symmetric}\Big\}.$$ 		
		Suppose $y\in S^k(u)$, then $y\in S^k_\ep(u)$ for some $\ep>0$. Thus, for any $(k+1)$-symmetric map $K\in L^2$ and $r\in(0,1)$, we have
		$$\medint_{B_1(0)}|u_{y,r}-K|^2dx\geq\ep.$$
		If $v$ is a tangent map of $u$ at $y$, then there exists a sequence $r_i\to 0$ such that $u_{y,r_i}\to v$ weakly in $H^1$ and strongly in $L^2$. This gives $$\medint_{B_1(0)}|K-v|^2dx\geq\ep,$$
		which implies that $v$ is not $(k+1)$-symmetric. Thus, $S^k(u)\subset \Sigma^k(u)$.
		
		Conversely, suppose $y\notin S^k(u)$. Then there exist sequences $r_i\in(0,1)$ and $(k+1)$-symmetric maps $K_i$ from $\R^{2m}$ to $N$ such that
		$$\medint_{B_1(0)}|u_{y,r_i}-K_i|^2dx\leq i^{-1}.$$
		We can select subsequences $\{u_{y,r_i}\} \rightharpoonup v$ weakly in $H^{1}$ and strongly in $L^2$, and $K_i\rightharpoonup K$ weakly in $L^2$. Then by the weak lower semi-continuity of $L^2$-norm, we obtain
		$$
		\int_{B_1(0)}|v-K|^2dx\leq\liminf_{i\to\oo}\int_{B_1(0)}|u_{y,r_i}-K_i|^2dx=0.
		$$
		Moreover, by the compactness of symmetric maps, $K$ is $(k+1)$-symmetric.
		
		If $r_i\to0$,  then $v$ is a tangent map and thus is $(k+1)$-symmetric, which shows that $y\not\in \Sigma^k(u)$. 
		
		If $\lim\limits_{i\to\oo}r_i=r>0$, we claim that $u(z)=K(z-y)$ for almost every $z\in B_\delta(y)$, where $\delta\in(0,r]$ is a positive radius. Indeed, we notice that
		$$
		\medint_{B_{r_i}(y)}|u(z)-K_i(z-y)|^2dz=\medint_{B_1(0)}|u_{y,r_i}(z)-K_i(z)|^2dz\leq i^{-1},
		$$
		which implies that there exists $\delta\in(0,r]$ such that
		$$
		\medint_{B_\delta(y)}|u(z)-K(z-y)|^2dz\le\liminf_{i\to\infty}\medint_{B_1(0)}|u_{y,r_i}(z)-K_i(z)|^2dz=0,
		$$
		using the lower semi-continuity of $L^2$-norm. Thus, all tangent $J$-holomorphic maps of $u$ at $y$ are $(k+1)$-symmetric and thus $y\notin \Sigma^k(u)$. The proof is complete.
	\end{proof}
	
	We recall the observation in \cite{Cheeger-Naber-2013-CPAM}. The following proposition will be used to establish a quantitative version of cone splitting later. 
	
	\begin{proposition}[Cone splitting principle]\label{prop: cone splitting}
		If $u: \mathbb{R}^{\ell_1}\rightarrow\mathbb{R}^{\ell_2}$ is $k$-symmetric with respect to a $k$-plane $V^k,$ and is $0$-symmetric at a point $z\not\in V^k$, then $u$ is $(k+1)$-symmetric with respect to the $(k+1)$-plane $V^{k+1}:=\Span\{z,V^k\}.$
	\end{proposition}

	Based on the cone splitting principle, the special structure of $J$-holomorphic maps allows us to prove the following symmetry improvement lemma.

	\begin{lemma}[Symmetry improvement]\label{lemma: symmetry improvement}
		Let $j_0$ be a constant complex structure on $\R^{2m}$, and $(N,J)$ be a smooth almost complex manifold. Suppose $u:(\R^{2m},j_0)\to (N,J)$ is a weakly $J$-holomorphic map. If $u$ is $(2k-1)$-symmetric for some $1\le k\le m$, then $u$ is $(2k)$-symmetric. 
	\end{lemma}
	
	\begin{proof}
		Suppose $u$ is $0$-homogeneous and translation invariant along a $(2k-1)$-dimensional subspace $V\subset\R^{2m}$. Together with the definition of $J$-holomorphic, we have for any $v\in V$, $$
		\partial_{j_0 v} u = du(j_0 v)= J(u) (\partial_v u) =0.
		$$
		This implies $u$ is translation invariant with respect to a new subspace $W:=V+j_0 V\subset \R^{2m}$. Since $j_0 W=W$, $\dim(W)$ must be even. Moreover, since $V\subset W$, we have $\dim(W)\ge 2k-1$ 
        { and hence} $\dim(W)\ge 2k$. The proof is complete.  
	\end{proof}

	Moreover, we have the following quantitative version of the  Lemma 3.6. 
	
	\begin{lemma}[Quantitative symmetry improvement]\label{lemma: quantitative symmetry improvement}
		{ For $\Lambda>0$,} let $u\in H^1_\La(M^{2m},N)$ be a stationary $J$-holomorphic map. For every $\e>0$, there exists $\delta=\delta(\e,M,N,g,h,j,J)>0$ such that if $u$ is $(2k-1,\delta)$-symmetric on $B_1(p)\subset M$ for some $1\le k\le m$, then $u$ is $(2k,\e)$-symmetric on $B_1(p)$.
	\end{lemma}
	
	\begin{proof}
		We prove by contradiction. Suppose there exist $\e_0>0$ and a sequence of stationary $J$-holomorphic maps $u_i\in H_\La^{1}(B_1(p),N)$ 
        { and } a sequence of $(2k-1)$-symmetric map $K_i\in L^2$ satisfying$$
		\medint_{B_1(p)} |u_i-K_i|^2\le \frac{1}{i},
		$$but $u_i$ is not $(2k,\varepsilon_0)$-symmetric. By Rellich's theorem, up to a subsequence if necessary, there exist maps $u$ and $K$ such that $u_i\wto u$ weakly in $H^1$ and strongly in $L^2$, $K_i\wto K$ weakly in $L^2$. By the weak lower semi-continuity of $L^2$-norm, we obtain
		$$
		\int_{B_1(p)}|u-K|^2dx\leq\liminf_{i\to\oo}\int_{B_1(p)}|u_i-K_i|^2dx=0.
		$$
		Moreover, by the compactness of symmetric maps, $K$ is $(2k-1)$-symmetric. Applying Lemma \ref{lemma: symmetry improvement}, we have $K$ is $(2k)$-symmetric. Then $\|u_i-K\|_{L^2(B_1(p))}\to 0$, which is a contradiction. 
	\end{proof}
	
	 The following symmetry improvement lemma for quantitative stratification of the singular set { follows immediately from Lemma 3.7}.
	
	
	\begin{lemma}[Quantitative symmetry improvement for the strata]\label{lemma: quantitative symmetry improvement stratification}
		{ For $\Lambda>0$,} let $u\in H^1_\La(M^{2m},N)$ be a stationary
		$J$-holomorphic map. Then for any $1\le k\le m$, there holds
		\begin{enumerate}
			\item For every $\e>0$, there exists
			$\delta=\delta(m,M,N,g,h,j,J,\La,\e)>0$ such that
			$$
			S^{2k-2}_{\e,r}(u)
			\subset S^{2k-1}_{\e,r}(u)
			\subset S^{2k-2}_{\delta,r}(u)
			$$
            for all $0<r<1$,
			and
            $$
			S^{2k-2}_{\e}(u)
			\subset S^{2k-1}_{\e}(u)
			\subset S^{2k-2}_{\delta}(u).
			$$
			\item $S^{2k-2}(u)=S^{2k-1}(u)$. Also, $$
			S^{0}(u)=S^{1}(u)\subset S^{2}(u)=S^{3}(u)\subset\cdots\subset S^{2m-2}(u)=S^{2m-1}(u)=\sing(u).
			$$
		\end{enumerate}
	\end{lemma}
	
	\begin{proof}
		For (i), suppose that $x\notin S^{2k-2}_{\delta,r}(u)$. Then there exists $s\in[r,1)$ such that $u$ is $(2k-1,\delta)$-symmetric on $B_s(x)$. By Lemma \ref{lemma: quantitative symmetry improvement}, we can select $\delta$ small such that $u$ is $(2k,\e)$-symmetric on $B_s(x)$. Hence $x\notin S^{2k-1}_{\e,r}(u)$. After taking the intersection over $r>0$ we prove (i).
		
		For (ii), suppose there is a $(2k-1)$-symmetric tangent $J$-holomorphic map $v$ of $u$ at $x\in M$. By Lemma \ref{lemma: symmetry improvement}, $v$ is $(2k)$-symmetric. Thus, $x\not\in S^{2k-2}(u)$ implies that $x\not\in S^{2k-1}(u)$. Since $S^{2k-2}(u)\subset S^{2k-1}(u)$, the proof is complete.
	\end{proof}

	Now we are ready to establish quantitative symmetry and quantitative cone splitting principle for stationary $J$-holomorphic maps. 
	
	\begin{proposition}\label{prop: quantitative symmetry}
		For any $\Lambda, \ep>0$, there exists $\delta_0=\delta_0(m,M,N,g,h,j,J,\ep,\Lambda)$ such that, if $u\in H^{1}_\La(M,N)$ is a stationary $J$-holomorphic map with
		$$\Phi_{u}(x,r)-\Phi_{u}(x,r/2)<\delta_0$$
		for some $x\in B_1(p)$ and $0<r<1$, then $u$ is $(0,\ep)$-symmetric on $B_r(x)$.
	\end{proposition}
	\begin{proof}
		We prove by contradiction. Suppose for some $\ep>0$ there is a sequence of stationary $J$-holomorphic map $u_i\in H^{1}_\La(M,N)$ satisfying
		$$\Phi_{u_i}(x_i,r_i)-\Phi_{u_i}(x_i,r_i/2)<i^{-1},$$
		but is not $(0,\ep)$-symmetric on $B_{r_i}(x_i)$.
		
		Let $\bar{u}_i(z)=u_i(x_i+r_iz)=u_i\circ\exp_{x_i}(r_iz)$. After selecting a subsequence, still denote by $\{\bar{u}_i\}$, we have $\bar{u}_i\rightharpoonup v$ weakly in $H^1$ and strongly in $L^2$. Then, by Theorem \ref{thm: defect measure}, $v$ is a weakly $J$-holomorphic map in $B_1$ with $\dim_{\mathcal{H}}\sing(v)\le 2m-2$. By the monotonicity formula \eqref{eq: monotonic}, for all $i\in \N$, we have
		$$
		\begin{aligned}
			\int_{B_1\backslash B_{1/2}}{|y\cdot D\bar{u}_i|^2}dy =&
			\int_{B_{r_i}\backslash B_{\frac{r_i}{2}}(x_i)}|(x-x_i)\cdot Du_i|^2r_i^{-2m}dx\\
			\le& C\int_{B_{r_i}\backslash B_{\frac{r_i}{2}}(x_i)}|x-x_i|^{2-2m}|\partial_\rho u_i|^2dx\\
			\le& C(\Phi_{u_i}(x_i,r_i)-\Phi_{u_i}(x_i,r_i/2))<Ci^{-1}.
		\end{aligned}
		$$
		Sending $i\to \infty$ and using the strong convergence of $\bar{u}_i\rightarrow v$ in $L^2$, we conclude that $v$ is radially invariant on $B_1\backslash B_{1/2}(0)$, and thus is $0$-homogeneous on $\R^{2m}$ by Proposition \ref{prop: unique continuation}.
		
		Finally, since $\bar{u}_i\rightarrow v$ strongly in $L^2$, we have$$
		\medint_{B_{r_i}(x_i)} |u_i(y)-v((y-x_i)/r_i)|^2dy=\medint_{B_{1}}|\bar{u}_i-v|^2\to 0\qquad \text{as }i\to \infty, 
		$$
		contradicting with the assumption that $u_i$ is not $(0,\ep)$-symmetric on $B_{r_i}(x_i)$. The proof is complete. 
	\end{proof}

	In order to extend Proposition \ref{prop: quantitative symmetry} to higher order symmetric cases, we introduce the notion of quantitative frame for affine subspaces as in \cite{Naber-V-2018}.
	
	\begin{definition}
		Let $\{y_i\}_{i=0}^k \subset B_1(p)\subset M$ and $\rho>0$. We say that these points $\rho$-effectively span a $k$-dimensional affine subspace if for all $i=1,\cdots,k$,
		$$\dist(z_i,z_0+\Span\{z_1-z_0,\cdots,z_{i-1}-z_0\})\geq2\rho,$$
		where $z_i=\exp^{-1}_p(y_i)\in\R^{2m}$ for $i=0,...,k$.
		More generally, a set $F\subset B_1(p)$ is said to $\rho$-effectively span a $k$-dimensional affine subspace, if there exist points $\{y_i\}_{i=0}^k\subset F$ which $\rho$-effectively spans a $k$-dimensional affine subspace.
	\end{definition}
	
	
	
	The following result extends Proposition \ref{prop: quantitative symmetry} to higher order quantitative symmetry.
	
	\begin{proposition}[Quantitative cone splitting]\label{prop: quant cone splitting}
		Fix $\Lambda, \ep,\rho>0$. Suppose $u\in H^{1}_{\Lambda}(M,N)$ is a stationary $J$-holomorphic map. Then there exists $\delta_1=\delta_1(m,M,N,g,h,j,J,\ep,\rho,\Lambda)>0$ such that if 
		\begin{enumerate}
		\item[{\upshape(a)}] $\Phi_{u}(x_s,1)-\Phi_{u}(x_s,1/2)<\delta_1 $ for some $\{x_s\}_{s=0}^k\subset B_1(p)$ and for all $i=0,\cdots,k$,
		
		\item[{\upshape(b)}] $\{x_s\}_{s=0}^k$ $\rho$-effectively spans a $k$-dimensional affine plane $V$,
		\end{enumerate}
		\noindent then $u$ is $(k,\ep)$-symmetric on $B_1(p)$.
	\end{proposition}

	\begin{proof}
		
		By translation if necessary, we may assume $p=0$. Suppose by contradiction that there is a sequence of stationary $J$-holomorphic maps $u^i\in H^{1}_\La(M,N)$, $\ep_0$, $\rho_0>0$,  $\{x_0^i,\cdots,x_k^i\}\subset B_1(0)$ and a sequence of $k$-dimensional affine plane $V^i$ such that $\{x_0^i,\cdots,x_k^i\}$  $\rho_0$-effectively spans $V^i$ and $\Phi_{u^i}(x^i_s,1)-\Phi_{u^i}(x^i_s,1/2)<\delta_i$ for $s=0,...,k$, where $\delta_i\to 0$ is chosen according to Proposition \ref{prop: quantitative symmetry}, so that $u^i$ is $(0,i^{-1})$-symmetric at $B_{1}(x_s^i)$ for each $s$. But $u^i$ is not $(k,\ep_0)$-symmetric in $B_1(0)$ for all $i\geq1.$ 
		
		Assume $u^i\rightharpoonup u$ weakly in $H^{1}$ and strongly in $L^2$, $x_s^i\rightarrow x_s\in B_1(0)$ for $s=0,...,k$ and $V^i\ri V$ as $i\rightarrow\infty.$
		Applying the monotonicity formula \eqref{eq: monotonic} to each $u^i$, we obtain
		$$\int_{B_1(x^i_s)\backslash B_{1/2}(x_s^i)}|(x-x^i_s)\cdot Du^i|^2\leq C\left(\Phi_{u^i}(x^i_s,1)-\Phi_{u^i}(x^i_s,1/2)\right)<C\delta_i.$$
		Sending $i\ri\oo$, we get
		$$\int_{B_1(x_s)\backslash B_{1/2}(x_s)}|(x-x_s)\cdot Du|^2=0.$$
		This shows that $u$ is radially symmetric on ${B_1(x_s)\backslash B_{1/2}(x_s)}$ for each $s=0,...,k$. The unique continuation property (see Proposition \ref{prop: unique continuation}) and the cone splitting principle (see Proposition \ref{prop: cone splitting}) imply that $u$ is symmetric with respect to the $k$-dimensional affine plane spanned by $\{x_s\}_{s=0}^k$. Since $u^i\to u$ strongly in $L^2$, this contradicts with the assumption that $u^i$ is not $(k,\ep_0)$-symmetric in $B_1(0)$.
	\end{proof}

	\subsection{Basic properties of quantitative stratum} 
	
	The following proposition shows that under  { suitable} conditions, the singular set satisfies the  one-side Reifenberg approximating property.
	\begin{proposition}\label{prop: stratum property-1} 
		For any $\Lambda, \ep,\rho>0$ and $B_3(p)\subset M$, there exists $\delta_2>0$ depending on $m,\Lambda,\ep,\rho,M,N,g,h,j,J$ such that if $u\in H^{1}_\La(M,N)$ is a stationary $J$-holomorphic map and the set $$
		F:=\{y\in  B_3(p): \Phi_u(y,2)-\Phi_u(y,\rho)<\delta_2\}$$
		$\rho$-effectively spans a $k$-dimensional affine plane $V$, then $$S^k_{\ep,\delta_2}(u)\cap B_1(p)\subset \exp_p(B_{2\rho}(V)).$$
	\end{proposition}

	To prove Proposition \ref{prop: stratum property-1}, we need the following lemma, which shows the relationship between almost translation invariant property and almost $(k+1)$-symmetric property.  { Henceforth}, we denote $|P\cdot Du|^2=\sum\limits_{i=1}^{k+1}|D_{e_i}u|^2$ with $\{e_i\}_{i=1}^{k+1}$ being an orthonormal basis of a $(k+1)$-dimensional subspace $P$.
	\begin{lemma}\label{lemma: almost translation invariant}
		For any $\Lambda, \ep>0$ and $B_3(p)\subset M$, there exists $\delta_3>0$, depending on $m,\Lambda,\ep,M,N,g,h,j,J$, such that if $u\in H^{1}_\La(M,N)$ is a stationary $J$-holomorphic map satisfying
		\begin{align}\label{eq: almost translation invariant}
			\int_{B_1(p)}|P\cdot D u|^2<\delta_3
		\end{align}
		for some $(k+1)$-dimensional subspace $P$, then 
		$S^k_{\ep,\overline{r}}(u)\cap B_{1/2}(p)=\emptyset$ for  $\overline{r}=\delta_3^{\frac{1}{2(2m-2)}}$. In particular, $0\not\in \exp^{-1}S^k_{\ep,\overline{r}}(u)$, i.e. $p\not\in S^k_{\ep,\overline{r}}(u)$.
	\end{lemma}
	The proof of Lemma \ref{lemma: almost translation invariant} is similar to Naber-Valtorta \cite[Lemma 4.8]{Naber-V-2018}. 
	\begin{proof}	
		Identify $B_3(p)\subset M$ with $B_3(0)\subset \R^{2m}$ via the exponential map, and still denote the pullback maps $\exp_p^*u$ and $\exp_p^*g$ by $u$ and $g$, respectively. We claim that there is a constant $C(m,\Lambda)>0$ such that for every $x\in B_{1/2}(p)$, there exists $r_x\in[\overline{r},1/2]$ such that
		\begin{equation}\label{eq: lemma 3-11-1}
			\Phi_u(x,r_x)-\Phi_u(x,r_x/2)<\frac{C(m,\Lambda)}{-\log\delta_3}.
		\end{equation}
		Indeed, if this is not true, then for $u\in H^{1}_\La(M,N)$ we may assume by choosing a good radius that $\Phi_u(x,1/2)\le C_0(m,\Lambda)$. Then 
		\begin{align*}
			C_0(m,\Lambda)\geq\Phi_u(x,1/2)&\geq\sum_{i=1}^{-\log\overline{r}+1}(\Phi_u(x,2^{-i})-\Phi_u(x,2^{-i-1}))\geq c(m)C(m,\Lambda),
		\end{align*}
		which is impossible if we let $C(m,\Lambda)=\frac{2C_0(m,\Lambda)}{c(m)}$. This proves the claim. 
		
		Suppose by contradiction that, there exist $\ep>0$ and a sequence of stationary $J$-holomorphic map $u_i\in H^{1}_\La(M,N)$ with subspaces $P_i$, $\de_{3,i}\to 0$, $x_i\in B_{1/2}(p)$ and $r_{x_i}\in[\bar{r}_i,1/2]$, such that $u_i$ is not $(k+1,\ep)$-symmetric on $B_{r_{x_i}}(x_i)$, \eqref{eq: lemma 3-11-1} holds for $x=x_i$, where $\bar{r}_i=\de_{3,i}^{\frac{1}{2(2m-2)}}$. Note also that by the definition of $r_{x_i}$, we have
		\begin{align}\label{eq: lemma 3-11-2}
			r_{x_i}^{2-2m}\int_{B_{r_{x_i}}(x_i)}|P_i\cdot Du_i|^2<r_{x_i}^{2-2m}\delta_{3,i}\leq\de_{3,i}^{1/2}.
		\end{align} 
		After a simple rotation, we may assume that the $(k+1)$-dimensional subspaces $P_i$ are equal to $P$ for all $i$. 
		
		Let $v_i(x)=u_i(x_i+r_{x_i}x)$. Then we may assume that $v_i$ converges weakly in $W^{1,2}$ and strongly in $L^2$ to $v$. By Corollary \ref{coro: compactness of SJholo}, $v$ is a weakly $J$-holomorphic map. It follows that $v$ is $0$-symmetric by unique continuation property (see proposition \ref{prop: unique continuation}) and invariant with respect to the $(k+1)$-dimensional subspace $P$ by \eqref{eq: lemma 3-11-2}: 
		$$
		0=\lim_{i\to \infty}r_{x_i}^{2-2m}\int_{B_{r_{x_i}}(x_i)}|P\cdot Du_i|^2\ge\liminf_{i\to \infty}\int_{B_1}|P\cdot Dv_i|^2\ge\int_{B_1}|P\cdot Dv|^2.
		$$
		This implies that $v_i$ is $(k+1,\ep)$-symmetric on $B_{1}$ for $i$  
        { large} enough, or equivalently, $u_i$ is $(k+1,\ep)$-symmetric with $B_{r_{x_i}}(x_i)$, which is a contradiction.
	\end{proof}

	Now we can prove Proposition \ref{prop: stratum property-1}.
	\begin{proof}[Proof of Proposition \ref{prop: stratum property-1}]
		Suppose $\{y_{i}\}_{i=0}^{k}\subset F$ is a $\rho$-independent frame
		that spans the affine subspace $V$, i.e., 
		$$
		V=\exp_p^{-1}y_{0}+{\rm span}\big\{\exp_p^{-1}y_{i}-\exp_p^{-1}y_{0}\big\}_{i=1}^{k}.
		$$
		Let $x_{0}\in B_{1}(p)\backslash \exp_pB_{2\rho}(V)$ and $\de_{2}>0$ to be
		determined later. Then we prove that $x_{0}\not\in S_{\ep,\de_{2}}^{k}(u)\cap B_{1}(p)$. 
		By the definition of $F$, there holds $\Phi_{u}(y_{i},2)-\Phi_{u}(y_{i},\rho)<\de_{2}$
		for all $0\le i\le k$. We can select a positive $r$ small enough such that  $B_{r}(x_{0})\subset B_{2}(y_{i})\backslash B_{\rho}(y_{i})$
		for every $i$. Thus, by the monotonicity formula \eqref{eq: monotonic}, for each $0\le i\le k$, it holds
		$$
		r^{2-2m}\int_{B_{r}(x_{0})}|(z-y_{i})\cdot Du|^{2}dz\le r^{2-2m}\int_{B_{2}(y_{i})\backslash B_{\rho}(y_{i})}|(z-y_{i})\cdot Du|^{2}dz\le C\de_{2}r^{2-2m}.
		$$
		Consequently, we obtain
		$$
		\begin{aligned}
			&r^{2-2m}\int_{B_{r}(x_{0})}|(y_{i}-y_{0})\cdot Du|^{2}\\
			\le&2r^{2-2m}\int_{B_{2}(y_{i})\backslash B_{\rho}(y_{i})}|(z-y_{i})\cdot Du|^{2}+2r^{2-2m}\int_{B_{2}(y_{i})\backslash B_{\rho}(y_{i})}|(z-y_{0})\cdot Du|^{2}\\
			\le&C_1\de_{2}r^{2-2m}.
		\end{aligned}
		$$
		Since $\{y_{i}\}_{i=0}^{k}$ is $\rho$-independent, we conclude that
		\begin{equation}\label{eq: hatV nabla u small}
			r^{2-2m}\int_{B_{r}(x_{0})}|\widehat{V}\cdot Du|^{2}\le C_2\de_{2}r^{2-2m},
		\end{equation}
		where {$\widehat{V}={\rm span}\{\exp_p^{-1}y_{i}-\exp_p^{-1}y_{0}\}_{i=1}^{k}$}.
		On the other hand, for each $z\in B_{r}(x_{0})\subset B_{2}\backslash B_{\rho}(V)$,
		let {$
			\pi_{V}(z)=\exp_p\left(\exp_p^{-1}y_{0}+\sum_{i=1}^{k}\al_{i}(z)(\exp_p^{-1}y_{i}-\exp_p^{-1}y_{0})\right)
			$} be the orthogonal projection of $z$ in $V$. Then $|z-\pi_{V}(z)|\ge\rho$,
		$|\al_{i}(z)|\le C(m,\rho)$ and
		\begin{align*}
		r^{2-2m}\int_{B_{r}(x_{0})}|(z-\pi_{V}(z))\cdot Du|^{2}dz&\le Cr^{2-2m}\sum_{i=0}^{k}\int_{B_{2}(y_{i})\backslash B_{\rho}(y_{i})}|(z-y_{i})\cdot Du|^{2}\\
        &\le C(m,\rho)\de_{2}r^{2-2m}.
		\end{align*}
		Thus, by setting $h(z)=(z-\pi_{V}(z))/|z-\pi_{V}(z)|$, it follows
		that
		$$
		\begin{aligned}
			&r^{2-2m}\int_{B_{r}(x_{0})}|h(x_{0})\cdot Du|^{2}\\
			\le& 2r^{2-2m}\int_{B_{r}(x_{0})}|h(z)\cdot Du|^{2}+2r^{2-2m}\int_{B_{r}(x_{0})}|(h(z)-h(x_{0}))\cdot Du|^{2}\\
			\le& C_3\de_{2}r^{2-2m}+C_3r^{2-2m}r^{2}\int_{B_{r}(x_{0})}|Du|^{2}\\
			\le& C_3\de_{2}r^{2-2m}+C_3\Lambda r^{2}.
		\end{aligned}
		$$
		Together with \eqref{eq: hatV nabla u small}, for the $(k+1)$-dimensional subspace $P:=\hat{V}\oplus \R h(x_0)$, we find that
		$$
		r^{2-2m}\int_{B_{r}(x_{0})}|P\cdot Du|^{2}\le 
		(C_2+C_3)\de_{2}r^{2-2m}+C_3\Lambda r^2=:C_4\de_{2}r^{2-2m}+C_3\Lambda r^2.
		$$
		Now we choose $r=r_0=\left(\frac{\delta_3}{3C_3\Lambda}\right)^{\frac12}\ll\rho$ and $\delta_2=\frac{\delta_3r_0^{2m-2}}{3C_4}$ such that $C_4\de_{2}r_0^{2-2m}=C_3\Lambda r_0^2=\frac{\de_3}{3}$.
		This gives
		$$
		r_0^{2-2m}\int_{B_{r_0}(x_{0})}|P\cdot Du|^{2}\le 2\delta_3/3.
		$$
		Then we can apply Lemma \ref{lemma: almost translation invariant} to conclude that $u$ is $(k+1,\ep)$-symmetric in $B_{r_0}(x_{0})$. This implies {$S^k_{\ep,\delta_2}(u)\cap B_1(p)\subset \exp_pB_{2\rho}(V)$}.
	\end{proof}

	The next result shows that $\Phi_u$ remains almost constant on all pinched points.
	\begin{lemma}\label{lemma: Phi_u almost const}
		Let ${\Lambda}, \rho,\eta>0$ be fixed and $u\in H^{1}_\La(M,N)$  { be} a stationary $J$-holomorphic map. Let  
		$$E:=\sup_{y\in B_2(p)}\Phi_u(y,1).$$
		Then there exists $\de_4=\de_4(m,\Lambda,\rho,\eta,M,N,g,h,j,J)>0$ such that if the set 
		$$F:=\{y\in B_2(p):\Phi_u(y,\rho)>E-\delta_4\}$$
		$\rho$-effectively spans a $k$-dimensional affine subspace $L\subset\mathbb{R}^{2m}$, then
		$$\Phi_u(x,\rho)\ge E-\eta \qquad \text{for all }  x\in \exp_p(L\cap B_2(0)).$$
	\end{lemma}
	\begin{proof}
		We prove by contradiction. { For $E>0$, }suppose $u_i\in H^{1}_\La(M,N)$ is a sequence of stationary $J$-holomorphic maps with
		$$\sup_{y\in B_2(p)}\Phi_{u_i}(y,1)\leq E.$$
		For each $i\ge 1$,  there exists $\{y^i_s\}_{s=0}^k\subset F_i:=\{y\in B_2(p):\Phi_{u_i}(y,\rho)>E-i^{-1}\}$   { that} $\rho$-effectively { spans} a $k$-dimensional affine subspace $L_i\subset\mathbb{R}^{2m}$, and there exists $x_i\in \exp_p(L_i\cap B_2(0))$ such that
		\begin{equation}\label{eq: Phi_u almost const-1}
			\Phi_{u_i}(x_i,\rho)\leq E-\eta.
		\end{equation}
		It follows from the assumption that
		\begin{equation}\label{eq: Phi_u almost const-2}
			\Phi_{u_i}(y^i_s,1)-\Phi_{u_i}(y^i_s,\rho)<1/i,\qquad \text{for all }\,i\ge 1,\, { s=0, 1,\cdots, k}. 
		\end{equation}
		Without loss of generality, we assume that for each $0\le s\le k$, 
		$$y^i_s\to y_s,\quad x_i\to x\quad\text{and}\quad L_i\to L$$
		as $i\to \infty$ with $x\in \exp_pL$ and $u_i\wto u$ in $W^{1,2}$ weakly.
		
		By Proposition \ref{thm: defect measure}, there exists a Radon measure $\mu=|\nabla u|^2dx+\nu$, where $u$ is a weakly $J$-holomorphic map and $\nu$ is the defect measure such that up to a subsequence $|\nabla u_i|^2dx\wto \mu$. By the monotonicity formula \eqref{eq: monotonic} and the stationary condition \eqref{eq: stationary condition} of $u_i$,  we find that for any fixed $y$, the function
		$\Phi_{\mu}(y,r):=\Phi_u(y,r)+r^{2-2m}\nu(B_r(y))$ is monotonically nondecreasing. Thus, as $i\to \infty$ in \eqref{eq: Phi_u almost const-2} and \eqref{eq: Phi_u almost const-1}, we obtain
		$$\Phi_{\mu}(y_s,1)=\Phi_{\mu}(y_s,\rho)=E,\qquad \text{for all }\, { s=0, 1,\cdots, k}$$
		and
		\begin{equation}\label{eq: loss of energy}
			\Phi_\mu(x,\rho)\le E-\eta.
		\end{equation}
		Moreover, as that of Proposition \ref{prop: tangent map}, we know that $\mu,u,\nu$ are translation invariant along $L$. Hence $\Phi_\mu|_{\exp_p(L\cap B_2(0))}\equiv {\Phi_\mu({y_0},1)}=E$, which contradicts with \eqref{eq: loss of energy}. The proof is complete.
	\end{proof}
	
	The following  result shows the almost symmetry under certain pinching condition.
	
	\begin{lemma}\label{lemma: pinching almost symmetry}
		For any ${\Lambda}, \ep,\rho>0$, there exists $\de_5=\de_5(m,\Lambda,\rho,\ep,M,N,g,h,j,J)>0$  {such that} if $u\in H^{1}_\La(M,N)$ is a stationary $J$-holomorphic map, $B_3(p)\subset M$ and  $\Phi_u(p,1)-\Phi_u(p,1/2)<\de_5$, and there is a point $y\in B_3(p)$ such that
		\begin{enumerate}
		\item [{\upshape (1)}] $\Phi_u(y,1)-\Phi_u(y,1/2)<\de_5$,
		
		\item[{\upshape (2)}] $u$ is not $(k+1,\ep)$-symmetric on $B_r(y)$ for some $r\in[\rho,2]$,    \end{enumerate}
		then $u$ is not $(k+1,\ep/2)$-symmetric on $B_r(p)$. 
    
	\end{lemma}
	In particular, under condition {\upshape (1)} of Lemma \ref{lemma: pinching almost symmetry}, {$y\in S^k_{\ep,\rho}(u)\cap B_3(p) \Rightarrow p\in S^k_{\ep/2,\rho}(u)$}.

	\begin{proof}
		Suppose by contradiction that there is a sequence $\{u_i\}\subset H^{1}_\La(M,N)$ of stationary $J$-holomorphic maps satisfying $\Phi_{u_i}(p,1)-\Phi_{u_i}(p,1/2)\leq i^{-1},$ and there exists a sequence $\{y_i\}\in B_3(p)$ such that $\Phi_{u_i}(y_i,1)-\Phi_{u_i}(y_i,1/2)\leq i^{-1}$ and that for each $i\in \N$, $u_i$ is not $(k+1,\ep)$-symmetric on $B_{r_i}(y_i)$, but it is $(k+1,\ep/2)$-symmetric on $B_{r_i}(p)$. 
		Then there exists a sequence of $(k+1)$-symmetric maps $K_i\in L^2(B_3(p), N)$ such that for any $i\in\N$,
		$$\medint_{B_{r_i}(p)}|u_i-K_i|^2\leq\ep/2.$$
		
		Up to a subsequence if necessary, we may assume that $y_i\ri y\in \overline{B}_3(p)$, $u_i\wto u$ weakly in $W^{1,2}$ and strongly in $L^2$, $r_i\to r\in[\rho,2]$ and $K_i\wto K$ in $L^2$. 
		
		By the unique continuation property (see proposition \ref{prop: unique continuation}), $u$ is a weakly $J$-holomorphic map and is homogeneous with respect to $p$ and $y$. By the compactness of symmetric functions, $K$ is also $(k+1)$-symmetric. Thus, we have
		$$\medint_{B_r(p)}|u-K|^2\leq\limsup_{i\ri\oo}\medint_{B_{r_i}(p)}|u_i-K_i|^2\leq\ep/2.$$
		Suppose $p=0$ for instance. If $y=0$, we get
		\begin{align*}
			&\lim_{i\ri\oo}\medint_{B_{r_i}(y_i)}|u_i(x)-K(x-y_i)|^2\\
			&\leq\lim_{i\ri\oo}\medint_{B_{r_i}(y_i)}|u_i(x)-K(x)|^2dx+\lim_{i\ri\oo}\medint_{B_{r_i}(y_i)}|K(x)-K(x-y_i)|^2dx\\
			&\leq\ep/2.
		\end{align*}
		If $y\neq0,$ then by the invariance of $K$ with respect to $\R y$, we get
		\begin{align*}
			&\lim_{i\ri\oo}\medint_{B_{r_i}(y_i)}|u_i(x)-K(x-y_i)|^2=\lim_{i\ri\oo}\medint_{B_{r_i}(0)}|u_i(y_i+x)-K(x)|^2dx\\
			&\leq 2\lim_{i\ri\oo}\medint_{B_{r_i}(0)}|u_i(y_i+x)-u(y+x)|^2dx+2\medint_{B_r(0)}|u(y+x)-K(x)|^2dx\\
			&{=2\medint_{B_r(y)}|u(x)-K(x-y)|^2dx=2\medint_{B_r(y)}|u(x)-K(x)|^2dx\le 2\cdot\frac{\ep}2=\ep}.
		\end{align*}
		In both cases, we  
        { obtain the desired} contradiction.	
	\end{proof}
	
	\begin{remark}\label{rmk:on Lemma 313}
		In Section \ref{sec:covering lemma}, we shall repeatedly use the following variant of Lemma \ref{lemma: pinching almost symmetry}: Suppose $u\in H^{1}_\La(M,N)$ is a stationary $J$-holomorphic map and $B_3(p)\subset M$ with $\Phi_u(p,1)-\Phi_u(p,1/2)<\de_5$. If there is some $y\in B_3(p)$ with $\Phi_u(y,1)-\Phi_u(y,1/2)<\de_5$, then $y\in S^k_{\ep,\rho}(u)\cap B_3(p)\Rightarrow p\in S^k_{\ep/2,\rho}(u)$.
	\end{remark}

	\section{Reifenberg theorems and estimates of Jones' number}\label{sec:reifenberg}
	In this section, we recall the Reifenberg type results obtained by Naber-Valtorta \cite{Naber-V-2017}. 
	We first recall the definition of Jones' number $\beta_2$, which quantifies how close the support of a measure $\mu$ is to a $k$-dimensional affine subspace.
	
	\begin{definition}\label{def: Jones number}
		Let $\mu$ be a nonnegative Radon measure on $B_3\subset \R^\ell$. Fix $k\in\N$. The $k$-dimensional Jones' $\beta_2$ number is defined as
		\begin{align*}
			\big(\beta^k_{2,\mu}(x,r)\big)^2:=\inf_{V\su\R^\ell}\int_{B_r(x)}\frac{d^2(y,V)}{r^2}\frac{d\mu(y)}{r^k},\qquad \text{when }\,B_r(x)\subset B_3,
		\end{align*}
		where the infimum is taken over all $k$-dimensional affine subspaces $V\subset\R^\ell$.
	\end{definition}
	
	We state two versions of the quantitative Reifenberg theorems from \cite{Naber-V-2017}.
	\begin{theorem}[{Discrete-Reifenberg, \cite[Theorem 3.4]{Naber-V-2017}}]\label{thm: discrete-Reifenberg}
		There exist two constants $\de_6=\de_6(m)$ and $C_R(m)$ such that the following property holds: Let $\{B_{r_x}(x)\}_{x\in \mathcal{C}}\su B_2(0)\su\R^m$ be a family of pairwise disjoint balls with centers in $\mathcal{C}\subset B_1(0)$ and let $\mu\equiv\displaystyle\sum_{x\in \mathcal{C}}\omega_kr_x^k\de_x$ be the associated measure. If for every ball $B_r(x)\su B_2,$ there holds
		\begin{align}\label{eq: multiscale appro-1}
			\int_{B_r(x)}\left(\int_0^r\beta^k_{2,\mu}(y,s)^2\frac{ds}{s}\right)d\mu(y)<\de_6^2r^k,
		\end{align}
		then we have the uniform estimate
		\begin{align*}
			\sum_{x\in \mathcal{C}}r^k_x< C_R(m).
		\end{align*}
	\end{theorem}

	\begin{theorem} [{Rectifiable-Reifenberg, \cite[Theorem 3.3]{Naber-V-2017}}]\label{thm: rectifiable-Reifenberg}
		There exist constants $\de_7=\de_7(m)$ and $C=C(m)$ satisfying the following property. Assume that $S\subset B_2\su\R^m$ is $\HH^k$-measurable, and for each $B_r(x)\subset B_2$ there holds
		\begin{align}\label{eq: multiscale appro-2}
			\int_{S\cap B_r(x)}\left(\int_0^r\beta^k_{2,\HH^k\lfloor S}(y,s)^2\frac{ds}{s}\right)d\HH^k(y)<\de_7^2r^k.
		\end{align}
		Then $S\cap B_1$ is  $k$-rectifiable, and   $\HH^k(S\cap B_r(x))\leq C r^k$ for each $x\in S\cap B_1$.
	\end{theorem}
	
	The key problem in application is how to establish an estimate for Jones' number $\beta_2$.  Similar to 
    { that} of harmonic maps in \cite{Naber-V-2017}, we need an $L^2$ subspace approximation theorem for stationary $J$-holomorphic maps.
	
	For $x\in B_1(p)$ and $r>0,$ we denote for brevity
	$$
	W_r(x):= W_{r,10r}(x)=\int_{B_{10r}(x)\ba B_r(x)}\frac{|(y-x)\cdot Du(y)|^2}{|y-x|^m}dy\geq0.
	$$
	for a stationary $J$-holomorphic map $u\in H^{1}(M,N)$. 
	
	The following result can be proved by the similar argument of Naber-Vatorta \cite{Naber-V-2017}. 
	\begin{theorem}\label{thm: L2 subspace approximation}
		Fix $\Lambda, \ep>0,0<r\leq1$ and $x\in B_1(p).$ There exist $C>0$ and $\delta_8>0$ depending on $m,\Lambda,\ep,M,N,g,h,j,J$ 
         such that, if $u\in H^{1}_\La(M,N)$ is a stationary $J$-holomorphic map, 
        { which} is $(0,\de_8)$-symmetric on $B_{10r}(x)$ but not $(k+1,\ep)$-symmetric, then for any nonnegative finite measure $\mu$ on $B_r(0)=\exp_p^{-1}(B_r(p))$ we have
		\begin{equation}\label{eq: L2 subspace approximation}
			\beta^k_{2,\mu}(x,r)^2\leq Cr^{-k}\int_{B_r(x)}W_r(y)d\mu(y).
		\end{equation}
	\end{theorem}
	
	Without loss of generality, we assume that $\mu$ is a probability measure supported on $B_1(0)=\exp_p^{-1}(B_1(p))$. Let $x_{cm}$ be the mass center of $\mu$ in $B_1(0)$, i.e.,  
	$
	x_{cm}\equiv\int xd\mu(x).
	$
	Here,  the integral domain is  $B_1(0)$ and is omitted for brevity. The second moment $Q=Q(\mu)$ of $\mu$ is the symmetric bilinear form defined by
	$$
	Q(v,w):=\int[(x-x_{cm})\cdot v][(x-x_{cm})\cdot w]d\mu(x),\qquad \text{for all }\ v,w\in \R^{2m}.
	$$
	Let $\lambda_1(\mu)\ge \cdots\ge \lam_{2m}(\mu)$ be  eigenvalues of $Q$ and $v_1(\mu),\cdots,v_{2m}(\mu)$ be the  { corresponding} eigenvectors. Then we have
	\begin{align}\label{eq: Q eigenvalue}
		Q(v_k)=\lam_kv_k=\int[(x-x_{cm})\cdot v_k][(x-x_{cm}) ]d\mu(x).
	\end{align}By the variational method, we can describe $\lam_1$ by $$\lam_1=\lam_1(\mu):=\max_{|v|=1}\int|(x-x_{cm})\cdot v|^2d\mu(x),$$
	where $v_1=v_1(\mu)$ is any unit vector achieving such a maximum. By induction we have
	$$
	\lam_{k+1}=\lam_{k+1}(\mu):=\max\left\{\int|(x-x_{cm})\cdot v|^2d\mu(x):{|v|^2=1}, v\perp\mbox{span}\{v_i\} \text{ for all }i\leq k\right\},
	$$
	and $v_{k+1}=v_{k+1}(\mu)$ is any unit vector achieving such a maximum. Note that by definition of $v_k$, we have  $V_k=x_{cm}+\mbox{span}\{v_1,\cdots,v_k\}$ is the $k$-dimensional affine subspace achieving the minimum in the definition of $\beta_2$ (see \cite[Remark 49]{Naber-V-2018}). Moreover,
	$$
	\beta_{2,\mu}^k(0,1)^2=\int d^2(x,V_k)d\mu(x)=\int\sum_{i=k+1}^{2m}((x-x_{cm})\cdot v_i)^2d\mu=\lam_{k+1}(\mu)+\cdots+\lam_{2m}(\mu).
	$$
	To prove Theorem \ref{thm: L2 subspace approximation}, we need the following property  on the relationship of $\lam_k,v_k$ and $W_1$. Here and after, let $A_{r_1,r_2}(a):=B_{r_2}(a)\backslash B_{r_1}(a)$ and $A_{r_1,r_2}:=B_{r_2}(0)\backslash B_{r_1}(0)$.
	
	\begin{proposition}[{\cite[Proposition 6.3]{Naber-V-2018}}]\label{prop: relationship of lam,v,W}
		Let $\mu$ be a probability measure in $B_1(0)$ and a map $u\in H^1(B_{10}(0), N)$ for some compact manifold $N$. Let $\lam_k$, $v_k$ be defined as above. Then there exists $C(m)>0$ such that
		\begin{equation}\label{eq: relationship of lam,v,W}
			\lam_k\int_{A_{3,4}}|v_k\cdot Du(z)|^2dz\leq C(m)\int_{B_1} W_1(x)d\mu(x), \quad \text{for all }k\ge 1,
		\end{equation}
	\end{proposition}
	
	To further estimate the left hand side of \eqref{eq: relationship of lam,v,W}, we need the following reversed result of Lemma \ref{lemma: almost translation invariant}.
	\begin{lemma}\label{lemma: non higher order sym}
		For any $\Lambda, \ep>0$, there exists  $\de_9=\de_9(m,\La,\ep,M,N,g,h,j,J)>0$ satisfying the following property: For any stationary $J$-holomorphic map $u\in H^{1}_\La(M,N)$, if it is $(0,\de_9)$-symmetric on $B_1(p)$ but not $(k+1,\ep)$-symmetric, then 	
		\begin{equation}\label{eq: non higher order sym}
			\int_{A_{3,4}}|P\cdot D u|^2>\delta_9
		\end{equation}
		for every $(k+1)$-dimensional subspace $P$.
	\end{lemma}
	\begin{proof}
		 Argue by contradiction.
        There is a sequence of stationary $J$-holomorphic maps $u_i\in H^{1}_\La(M,N)$ such that $u_i$ is $(0,i^{-1})$-symmetric on $B_1(p)$ but not $(k+1,\ep_0)$-symmetric for some $\ep_0>0$ fixed. Moreover, after an orthogonal transformation, there is a  $(k+1)$-dimensional subspace $P$ such that
		$$
		\int_{A_{3,4}}|P\cdot D u_i|^2\leq i^{-1},\ \forall i.
		$$
		Up to a subsequence if necessary, we assume $u_i\wto u$ weakly in $H^{1}$ and strongly in $L^2$ for some 0-homogeneous $J$-holomorphic  map $u\in H^1$ and 
		\begin{align*}
			\int_{A_{3,4}}|P\cdot Du|^2=0.
		\end{align*}
		By the unique continuation property we know 
		$	\int_{B_4}|P\cdot Du|^2=0.$ 
		Thus $u$ is $(k+1)$-symmetric on $B_1(0)$. Consequently, $u_i$ is $(k+1,\ep_0)$-symmetric for $i\gg 1$, since $u_i$ converges strongly to $u$ in $L^2$. This is a contradiction.  The proof is complete. 
	\end{proof}
	
	Now we prove Theorem \ref{thm: L2 subspace approximation}.
	\begin{proof}[Proof of Theorem \ref{thm: L2 subspace approximation}]  
		We denote $B_1(0)$ for $\exp_p^{-1}(B_1(p))$. Without loss of generality, we assume $\mu(B_1(0))=1$. Since $\lambda_k$ is nonincreasing,
		\begin{equation}\label{eq: beta<(2m-k)lam-k+1}
			\beta_{2,\mu}^k(0,1)^2=\lam_{k+1}+\cdots+\lam_{2m}\leq(2m-k)\lam_{k+1}.
		\end{equation}
		Thus it suffices to estimate $\la_{k+1}$. 
		Let $V^{k+1}=\mbox{span}\{v_1,\cdots,v_{k+1}\},$ then
		\begin{align*}
			&\lam_{k+1}\int_{A_{3,4}}|V^{k+1}\cdot Du(z)|^2dz=\lam_{k+1}\sum_{i=1}^{k+1}\int_{A_{3,4}}|Du(z)\cdot v_i|^2dz\\
			\le& \sum_{i=1}^{k+1}\lam_i\int_{A_{3,4}}|Du(z)\cdot v_i|^2dz\le C\int W_1d\mu(x),
		\end{align*}
		where the second inequality is due to Proposition \ref{prop: relationship of lam,v,W}.
		Set $\delta_8=\delta_9$, where $\delta_9$ is the number defined in Lemma \ref{lemma: non higher order sym}. By Lemma \ref{lemma: non higher order sym}, there holds 
		\begin{align*}
			\int_{A_{3,4}}|V^{k+1}\cdot Du(z)|^2dz\geq \de_8,
		\end{align*}
		since by assumption $u$ is $(0,\de_8)$-symmetric but not $(k+1,\ep)$-symmetric on $B_{10}(p)$. Then we have 
		\begin{align*}
			\de_8\lam_{k+1}\leq\lam_{k+1}\int_{A_{3,4}}|V^{k+1}\cdot Du(z)|^2dz \leq C\int W_1d\mu(x).
		\end{align*}
		Since $\de_8=\de_9$ depends only on $m,\La,\ep,M,N,g,h,j,J$, by \eqref{eq: beta<(2m-k)lam-k+1}, we conclude 
		\begin{align*}
			\beta_{2,\mu}^k(0,1)^2\leq C\int W_1d\mu(x).
		\end{align*}
		This completes the proof.
	\end{proof}

	\section{Covering lemma}\label{sec:covering lemma}
	In this section, we establish a covering Lemma which aims to give an estimate of the energy at singular points. The original idea is from \cite[Section 6.2]{Naber-V-2018} and here we shall follow the slightly modified presentation in \cite{GJXZ-2024} and \cite{Guo-Liu-Xiang-2026}. 
	
	Our main result of this section is to prove the following covering Lemma. 
	
	\begin{lemma}[Main covering Lemma]\label{lemma: main covering lemma}
		{ For $\Lambda>0$,} let $u\in H^1_\La(M,N)$ be a stationary $J$-holomorphic map and $B_3(p)\subset M$. For any $\ep>0$ and $0<r<R\leq1$, there exist constants $\delta=\de(m,\La,\ep,M,N,g,h,j,J)>0$ and $C=C(m,\La,\ep,M,N,g,h,j,J)$ with the following property:
		For any ${S}\su S^k_{\ep,\de r},$ there exists a finite covering of $S\cap B_R(p)$ such that
		\begin{equation}\label{eq: main cover property}
			S\cap B_R(p)\su \bigcup_{x\in\mathcal{C}}B_{r_x}(x) \quad\text{with}\quad  r_x\geq r \quad\text{and}\quad  \sum_{x\in\mathcal{C}}r_x^k\leq C R^k.
		\end{equation}
		Moreover, the balls in $\{B_{r_x/5}(x)\}_{x\in\cal{C}}$ are pairwise disjoint and $\cal{C}\subset S\cap B_R(p) $. 
	\end{lemma}
	
	\subsection{Proof of Lemma \ref{lemma: main covering lemma}}
	
	To prove Lemma \ref{lemma: main covering lemma}, we need two  Lemmas due to Naber-Valtorta \cite{Naber-V-2018}; see also \cite{Guo-Liu-Xiang-2026} for detailed proofs. 
	
	\begin{lemma}[Covering Lemma I]\label{lemma: covering lemma I}
		{ For $\Lambda>0$,} let $u\in H^1_\La(M,N)$ be a stationary $J$-holomorphic map and $B_3(p)\subset M$. For any $\ep>0$, $0<\rho=\rho(m)\leq\frac{1}{100}$ and $0<r<R\leq1$, there exist constants $\de=\de(m,\La,\rho,\ep,M,N,g,h,j,J)>0$ and $C_1(m)$ with the following property: 
		For any ${S}\su S^k_{\ep,\de r}$ with
		$$E:=\sup\limits_{x\in B_{2R}(p)\cap S}\Phi_u(x,R)\leq\Lambda, $$ 
		there exists a finite covering of $S\cap B_R(p)$ such that
		\begin{equation}\label{eq: covering lemma I}
			S\cap B_R(p)\su \bigcup_{x\in\mathcal{C}}B_{r_x}(x), \quad\text{with}\quad  r_x\geq r \quad\text{and}\quad  \sum_{x\in\mathcal{C}}r_x^k\leq C_1(m)R^k.
		\end{equation}
		Moreover, for each $x\in\mathcal{C}$, one of the following conditions is satisfied:
		\begin{itemize}
			\item[(i)] $r_x=r$;
			\item[(ii)] the set of points $F_x:=\big\{y\in S\cap B_{2r_x}(x):\Phi_u(y,\rho r_x/10)>E-\de\big\}$ is contained in $B_{\rho r_x/5}(L_x)\cap B_{2r_x}(x),$ where $L_x$ is  a  $(k-1)$-dimensional affine subspace.
		\end{itemize}
	\end{lemma}

	\begin{lemma}[Covering Lemma II]\label{lemma: covering lemma II}
		{ For $\Lambda>0$}, let $u\in H^1_\La(M,N)$ be a stationary $J$-holomorphic map and $B_3(p)\subset M$. For any $\ep>0$ and $0<r<R\leq1$, there exist $\de=\de(m,\La,\ep,M,N,g,h,j,J)>0$ and $C_2(m)$ such that for any subset ${S}\su S^k_{\ep,\de r}$, there exists a finite covering of $S\cap B_R(p)$ by
		\begin{equation}\label{eq: covering lemma II-1}
			S\cap B_R(p)\su \bigcup_{x\in\mathcal{C}}B_{r_x}(x), \quad\text{with}\quad \ r_x\geq r\quad\text{and}\quad\sum_{x\in\mathcal{C}}r_x^k\leq C_2(m)R^k.
		\end{equation}
		Moreover, for each $x\in\mathcal{C}$, one of the following conditions is satisfied:
		\begin{itemize}
			\item[i)] $r_x=r;$
			\item[ii)] we have the following uniform energy drop property:
			\begin{equation}\label{eq: covering lemma II-2}
				\sup_{y\in B_{2r_x}(x)\cap S}\Phi_u(y,r_x)\leq E-\de,
			\end{equation} 
			where $$E:=\sup_{x\in B_{2R}(p)\cap{S}}\Phi_u(x,R).$$
		\end{itemize}
	\end{lemma}

	Based on the above two Lemmas, we can complete the proof of Lemma \ref{lemma: main covering lemma}.
	\begin{proof}[Proof of Lemma \ref{lemma: main covering lemma}]
		Note that the energy $E$ defined as in Lemma \ref{lemma: covering lemma II} satisfies $E\leq C\Lambda$. So iterating Lemma \ref{lemma: covering lemma II} by at most $i=[\de^{-1}E]+1$ times, where $i$ depends on $m,\Lambda,\e,M,N,g,h,j,J$ and $E$, we could obtain a covering $\{B_{r_x}(x)\}_{x\in \mathcal{C}^i}$ of $S\cap B_R(p)$ such that $r_x\ge r$ and 
		$$
		\sum_{x\in \cal{C}^i}r_x^k\leq C_2(m)^iR^k.
		$$ 
		Then for $x\in S\cap B_R(p)$, we have the larger covering $\{B_{2r_x}(x)\}_{x\in \widetilde{\cal C}^i}$ such that
		$$
		S\cap B_R(p) \subset \bigcup_{x\in \widetilde{\cal C}^i} B_{2 r_x}(x)\quad \text{ and } \quad  \sum_{x\in \cal{C}}(2 r_x)^k\leq 2^kC_2(m)^iR^k.
		$$
		Finally, by Vitali's covering Lemma, we can  select a family of disjoint balls $\{B_{2r_x}(x)\}_{x\in \cal{C}}$ from  $\widetilde{\cal C}^i$ such that
		$$ S\cap B_R(p) \subset \bigcup_{x\in \cal{C}} B_{10r_x}(x)\quad \text{ and } \quad  \sum_{x\in \cal{C}}(10r_x)^k\leq 10^kC_2(m)^iR^k.$$
		The proof is complete by taking $C(m,\Lambda,\e,M,N,g,h,j,J)=10^kC_2(m)^i$. 
	\end{proof}

	\section{Proofs of main theorems}\label{sec:proofs}
	
	\subsection{Proof of Theorem \ref{thm: stratification of SJholo}}
	Theorem \ref{thm: stratification of SJholo} follows from the main covering Lemma, Lemma \ref{lemma: main covering lemma}, and the rectifiable Reifenberg theorems in Section \ref{sec:reifenberg}.
	\begin{proof}[Proof of Theorem \ref{thm: stratification of SJholo}]
		
		Let $\de=\de(m,\Lambda,\ep,M,N,g,h,j,J)>0$ be fixed. By Lemma \ref{lemma: main covering lemma}, for all $0<r<\delta$,
		\begin{equation}\label{eq: main thm proof-1}
			\mbox{Vol}(T_r(S^k_{\ep,r}(u))\cap B_1(p))\le \mbox{Vol}\Big(T_r\Big(S^k_{\ep,\de }(u)\cap B_1(p)\Big)\Big) \leq C'_\ep r^{2m-k}
		\end{equation}
		with $C'_{\ep}=C'_\ep(m,\Lambda,\ep,M,N,g,h,j,J)$. For $r\in[\de,1)$, we have 
		\begin{equation}\label{eq: main thm proof-2}
			\mbox{Vol}\Big(T_r\Big(S^k_{\ep,r}(u)\cap B_1(p)\Big)\Big) \le \mbox{Vol}(T_1(B_1(p))) \leq C_m \le C_m \left(\frac{r}{\de} \right)^{2m-k}.
		\end{equation}
		\eqref{eq: main thm proof-1} and \eqref{eq: main thm proof-2} complete the proof of  \eqref{eq: Vol estomate-1} for all $0<r<1$.  The volume estimate \eqref{eq: Vol estomate-2} follows from \eqref{eq: Vol estomate-1} by noting that $S^k_{\ep}(u)\subset S^k_{\ep, r}(u)$ for any $r>0$.

		In order to prove the rectifiability of $S^k(u)$, it is sufficient to prove the rectifiability of $S^k_\ep(u)$ for each $\ep>0$, since $S^k(u)=\bigcup\limits_{i\ge 1}S^k_{1/i}(u)$.
		
		By the volume estimate \eqref{eq: Vol estomate-2}, we have $\HH^k(S^k_\ep(u)\cap B_1(p))\leq C_\ep.$ Similarly, for $B_r(x)$ with $x\in B_1(p)$ and $r\leq 1$, we have
		\begin{equation}\label{eq: main thm Alhfors}
			\HH^k(S^k_\ep(u)\cap B_r(x))\leq C_\ep r^k.
		\end{equation}
		
		Let $S\su S_\ep^k(u)\cap B_1(p)$ be an arbitrary measurable subset with $\HH^k(S)>0$. Set 
		$$g(x,r)=\Phi_u(x,r)-\Phi_u(x,0),\quad x\in B_1(p),\ r\leq 1.$$
		This function is nondecreasing in $r$ and uniformly bounded for all $x\in B_1(p)$ and $r\in[0,1]$, and pointwise converging to $0$ as $r\to0$. Then the dominated convergence theorem implies that for all $\de>0$, there exists $\overline{r}>0$ such that
		\begin{equation}\label{eq: select E}
			\medint_S g(x,10\overline{r})d\HH^k(x)\leq\de^2.
		\end{equation}
		So we can find a measurable subset $E\su S$ with $\HH^k(E)\leq\de\HH^k(S)$ and $g(x,10\overline{r})\leq\de$ for all $x\in F:=S\ba E$. Now cover $F$ by a finite number of balls $B_{\overline{r}}(x_i)$ centered on $F$. Rescaling if necessary, we may assume that $B_{\overline{r}}(x_i)=B_1(0)$. Then $g(x,10)\leq\de$ for $x\in F$. Similar to \eqref{eq: select E}, we can choose $\de$ sufficiently small such that $u$ is $(0,\de_9)$ symmetric in $B_{10}$. Theorem \ref{thm: L2 subspace approximation} implies 
		\begin{equation*}
			\beta_{2,\HH^k\lfloor F}(z,s)^2\leq C_1s^{-k}\int_{B_s(z)}W_s(y)d\HH^k\lfloor F(y) \quad \text{for all }z\in F,\ s\leq1.
		\end{equation*}
		Integrating the previous estimate with respect to $z$ and using \eqref{eq: main thm Alhfors} yield that for all  $x\in B_1(p)$ and $s\leq r\leq1,$ there holds
		\begin{align*}
			\int_{B_r(x)}\beta_{2,\HH^k\lfloor F}(z,s)^2d\HH^k\lfloor F(z)&\leq C_1s^{-k}\int_{B_r(x)}\int_{B_s(z)}W_s(y)
            d\HH^k\lfloor F(y)d\HH^k\lfloor F(z)\\
			&\leq C_1C_\ep\int_{B_{r+s}(x)}W_s(z)d\HH^k\lfloor F(z).
		\end{align*}
		Integrating again with respect to $s$, we obtain that for all $x\in B_1(p)$ and $r\leq1,$
		\begin{align*}
		\int_{B_r(x)}\int_0^r\beta_{2,\HH^k\lfloor F}(z,s)^2\frac{ds}{s}d\HH^k\lfloor F(z)&\leq C_1C_\ep\int_{B_{2r}(x)}[\Phi_u(z,10r)-\Phi_u(z,0)]d\HH^k\lfloor F(z)\non\\
			&\leq c(m)C_1C_\ep^2\de r^k.    
		\end{align*}
		By choosing $\de\le\frac{\de_7^2}{c(m)C_1C_\ep^2}$,
		we deduce from Theorem \ref{thm: rectifiable-Reifenberg} that $F\cap B_1(p)$ is $k$-rectifiable.
		
		Repeating the above argument with $E$ in place of $F$, we could find another measurable set $E_1\subset E$ with $\HH^k(E_1)\leq \delta \HH^k(E)$, and that $F_1:=E\ba E_1$ is $k$-rectifiable. Continuing this process, we eventually conclude that $S$ is $k$-rectifiable. 
		
		The proofs of final assertions are similar to that of \cite{Naber-V-2017} and are thus omitted. Combined with quantitative symmetry improvement lemma \ref{lemma: quantitative symmetry improvement stratification}, the proof is complete.
	\end{proof}
	
	\subsection{Proof of Theorem \ref{thm: N no holo S^2 implies u smooth}}
	
	\begin{proof}[Proof of Theorem \ref{thm: N no holo S^2 implies u smooth}]
		 We use an induction  { on $m$}.  { For $m=2$, suppose $\sing(u)\not=\emptyset$. Let $x_0\in\sing(u)$, since there is no non-constant $J$-holomorphic maps from $\mathbb{S}^2$ to $N$, there exists
        a nontrivial stationary tangent map $\omega(z)=\omega(\frac{z}{|z|}):\mathbb{C}^2\to N$ of $u$ at $x_0$
        such that $z_0\in\sing(\omega)$ for some
        $z_0\in \mathbb{S}^3\subset \mathbb{C}^2$.
        From the equation of $J$-holomorphic maps,
        $w_0=j_0(z_0)\in\sing(\omega)$. Thus
        $\sing(\omega)\supset {\rm{span}}\{z_0, w_0\}$
        and $\mathcal{H}^2(\sing(\omega))>0$. This is impossible. Thus the conclusion holds for $m=2$}.
        Assume that for all $2\le m\le \ell-1$, the conclusion holds. Now we prove the case when $m=\ell$. 
		Suppose  $\sing(u)\not=\emptyset$, then by Proposition \ref{prop: tangent map},  for $x\in\sing(u)$, there exists a non-constant tangent $J$-holomorphic map $T:(\R^{2\ell},j_0)\to N$ of $u$ at $x$, where $j_0$ is the standard complex structure on $\R^{2\ell}$. By Corollary \ref{coro: compactness of SJholo}, $T$ is a stationary $J$-holomorphic map. Since $T$ is $0$-homogeneous by Proposition \ref{prop: tangent map}, we have $$
		\frac{\partial T}{\partial r}=0\quad \text{for a.e. }x\in \R^{2\ell}\backslash\{0\}.
		$$
		By the definition of $J$-holomorphic map, we also have$$
		dT(j_0\frac{\partial}{\partial r})=J(T)\circ dT(\frac{\partial}{\partial r})=0\quad \text{for a.e. }x\in \R^{2\ell}\backslash\{0\}.
		$$
		Therefore, there exists a non-constant stationary $J$-holomorphic map $\phi:S^{2\ell-1}/S^1\cong\CP^{\ell-1}\to N$. Since $\CP^{\ell-1}$ has complex dimension $\ell-1$, the induction hypothesis implies that $\phi$ is smooth. Thus, we can choose $a,b\in\mathbb{CP}^{\ell-1}$ such that $\phi(a)\ne\phi(b)$, then we let $L\cong\CP^1$ be the projective line through $a$ and $b$. Since $L$ is a complex submanifold of $\mathbb{CP}^{\ell-1}$, we claim that $\phi|_L:L\cong\CP^1\longrightarrow N$ is a non-constant smooth $J$-holomorphic map, which is a contradiction. The proof is complete.			
	\end{proof}

	
	\bigskip 
	\noindent\textbf{Acknowledgements.} Guo and Liu are supported by the Young Scientist Program of the Ministry of Science and Technology of China (No.~2021YFA1002200), the National Science Foundation of China (No.~12671095), the Taishan Scholar Project and the Jiangsu Provincial Scientific Research Center of Applied Mathematics under Grant No. 7707014040A. Wang is partially supported by NSF DMS 2453789 and Simons Travel Grant TSM-00007723. 
	
	\medskip
	
	\textbf{Data Availability Statement}. Data sharing is not applicable to this article as no
	datasets were generated or analysed during the current study.
	\medskip
	
	\textbf{Conflict-of-Interest Statement.}	The authors declare no conflict of interest.

\end{document}